\documentclass[11pt,reqno]{amsart}
\usepackage{graphicx,color}
\usepackage{amssymb,amscd,latexsym, mathrsfs, epsfig}
\usepackage[all]{xy}
\usepackage[latin1]{inputenc}

\newtheorem{thm}{Theorem}[section]
\newtheorem{lem}[thm]{Lemma}
\newtheorem{pro}[thm]{Proposition}
\newtheorem{rem}[thm]{Remark}

\newtheorem{kor}[thm]{Corollary}
\newtheorem{defi}[thm]{Definition}

\newcommand\PP{\mathbb P}
\newcommand\C{\mathbb C}
\newcommand\Q{\mathbb Q}
\newcommand\Qn{\mathbb Q}
\newcommand\R{\mathbb R}
\newcommand\Z{\mathbb Z}
\newcommand\Zn{\mathbb Z}
\newcommand\N{\mathbb N}
\newcommand\Nn{\mathbb{N}}

\newcommand{\Aut}{\operatorname{Aut}}
\newcommand{\mult}{\operatorname{Mult}}

\title[]{MPS degeneration formula for quiver moduli and refined GW/Kronecker correspondence}
\author{M. Reineke, J. Stoppa and T. Weist}
\date{\today}
\address{}
\email{}
\begin{document}
\begin{abstract} Motivated by string-theoretic arguments Manschot, Pioline and Sen discovered a new remarkable formula for the Poincar\'e polynomial of a smooth compact moduli space of stable quiver representations which effectively reduces to the abelian case (i.e. thin dimension vectors). We first prove a motivic generalization of this formula, valid for arbitrary quivers, dimension vectors and stabilities. In the case of complete bipartite quivers we use the refined GW/Kronecker correspondence between Euler characteristics of quiver moduli and Gromov-Witten invariants to identify the MPS formula for Euler characteristics with a standard degeneration formula in Gromov-Witten theory. Finally we combine the MPS formula with localization techniques, obtaining a new formula for quiver Euler characteristics as a sum over trees, and constructing many examples of explicit correspondences between quiver representations and tropical curves. 
\end{abstract}
\maketitle 
\section{Introduction}

\noindent In \cite{mps}, J. Manschot, B. Pioline and A. Sen derive a remarkable formula for the Poincar\'e polynomial of a smooth compact moduli space of stable quiver representations (called MPS degeneration formula in the following), motivated by string-theoretic techniques (more precisely an interpretation relating quiver moduli to multi-centered black hole solutions to $N=2$ supergravity). In contrast to the previously available formulae \cite{ReHNS} expressing the Poincar\'e polynomial explicitly using (a resolution of) a Harder-Narasimhan type recursion, the MPS degeneration formula expresses it as a summation over Poincar\'e polynomials of moduli spaces of several other quivers, but only involving very special (thin, i.e. type one) dimension vectors (in the language of \cite{mps} the index of certain non-abelian quivers without oriented loops can be reduced to the abelian case, by a physical argument which allows trading Bose-Fermi statistics with its classical limit, Maxwell-Boltzmann statistics). One immediate advantage is that the MPS formula specializes to a similar formula for the Euler characteristics, which is not possible for the Harder-Narasimhan recursion. Surprisingly, the derivation of the MPS formula in \cite[Appendix D]{mps} relies completely on the resolved Harder-Narasimhan recursion of \cite{ReHNS}. One should note however that even in the very special cases when Euler characteristics of quiver moduli were already known, the MPS degeneration formula derives these numbers in a highly nontrivial way.\\[1ex]

As a first result of the present work we prove a motivic generalization of the MPS degeneration formula (Theorem \ref{motivicMPS}) which is meaningful for arbitrary quivers, dimension vectors and stabilities. Essentially, the motivic MPS formula expresses the motive of the quotient stack of the locus of semistable representations by the base change group in terms of similar motives for thin dimension vectors of a covering quiver, as an identity in a suitably localized Grothendieck ring of varieties. The proof essentially proceeds along the lines of \cite[Appendix D]{mps}, but avoids the resolution formula for the Harder-Narasimhan recursion, and clarifies the role of symmetric function identities implicit in \cite{mps}. Specialization of the motivic identity to Poincar\'e polynomials recovers a generalization of the formula of \cite{mps}, which is now shown to hold for arbitrary quivers, arbitrary stabilities and coprime dimension vectors. We also derive a dual MPS degeneration formula (Corollary \ref{dualMPS}), which has the advantage of reducing to a smaller covering quiver, at the expense of having more general dimension vectors.\\[2ex]

In Section \ref{4} we take up a second line of investigation, connected with the so-called GW/Kronecker correspondence based on \cite{gps}, \cite{gp} or more precisely its refinement described in \cite{rw}. A typical result of this type states that the Euler characteristic of certain moduli spaces of representations for suitable quivers (e.g. generalized Kronecker quivers) can be computed alternatively as a Gromov-Witten invariant (on a weighted projective plane). In particular this is the case for coprime dimension vectors of complete bipartite quivers, to which we restrict throughout Section \ref{4}. Writing down the MPS degeneration formula for quiver Euler characteristics in this context one notices a striking similarity with the degeneration formulae which are commonly used in Gromov-Witten theory, expressing a given Gromov-Witten invariant in terms of \emph{relative} invariants, with tangency conditions along divisors. Theorem \ref{mpsdeg} puts this intuition on firm ground: at least for coprime dimension vectors of bipartite quivers, the MPS formula is indeed completely equivalent to a much more standard degeneration formula in Gromov-Witten theory. The proof hinges on the equality of certain Euler characteristics with tropical counts (Proposition \ref{eulgw}).\\

Combining localization techniques with the MPS formula leads to the remarkable conclusion that the Euler characteristic of moduli spaces of stable representations is obtained in a purely combinatorial way. Indeed, since the dimension vectors considered after applying the MPS formula are of type one, the moduli spaces corresponding to torus fixed representations are just isolated points, so that every such moduli space corresponds to a tree with a fixed number of (weighted) points. We describe this method for bipartite quivers in Section \ref{trees} (see especially Corollary \ref{treeSum}), but it could easily be transferred to general quivers without oriented cycles.\\

In Section \ref{loctrop} we analyse the identity of Euler characteristics with tropical counts found in Proposition \ref{eulgw} from this point of view. On the one hand with a pair of weight vectors $({\bf w}(k^1) , {\bf w}(k^2))$ we can associate a tropical curve count $N^{\rm trop}[({\bf w}(k^1) , {\bf w}(k^2))]$, which effectively counts suitable trees. On the other hand with the same weight vector we can associate a quiver Euler characteristic $\chi(M^{\Theta_l-\rm st}_{(k^1,k^2)}(\mathcal{N}))$, which by the above argument (MPS plus localization) is also enumerating certain trees. 

Thus one would expect to be able to find an explicit way of assigning a quiver localization data to one of our tropical curves, and vice versa. The analogy between quiver localization data and tropical curves was already pointed out in \cite{s}, but the MPS formula makes it even stronger. Notice that both tropical curves and localization data naturally carry multiplicities: for curves this is the standard tropical multiplicity (recalled in Section \ref{4}), while in the case of quiver moduli spaces, a fixed tree can be coloured in different ways to obtain a number of torus fixed points. The first natural guess is that the number of colourings and the multiplicity of some corresponding tropical curve coincide. Unfortunately this doesn't work, simply because in general the numbers of underlying curves and trees (forgetting the multiplicity) are different. 

The next more promising attempt is described in Section \ref{loctrop}. On both sides there is a way to construct new combinatorial data recursively. On the one hand we show that our tropical curves of prescribed slope can be obtained by glueing smaller ones in a unique way (at least when the set of prescribed, unbounded incoming edges is chosen generically). This construction also gives a recursive formula for the tropical counts, see Theorem $\ref{troprec}$. On the other hand we have a similar construction for quiver localization data. Starting with a number of semistable tuples, i.e. tuples consisting of a tree and a dimension vector such that the corresponding moduli space of semistables is not empty, we can glue them in a similar way to obtain a localization data with greater dimension vector, see Theorem \ref{glue}.

In many cases these recursive constructions lead to a direct correspondence between tropical curves and quiver localization data. In Section \ref{examples} we describe two such families of examples in detail. We do not know at the moment a method which gives a concrete geometric correspondence in full generality.\\

\noindent\textbf{Acknowledgements.} We are grateful to So Okada for drawing our attention to the formula of Manschot, Pioline and Sen. This research was partially supported by Trinity College, Cambridge. Part of this work was carried out at the Isaac Newton Institute for Mathematical Sciences, Cambridge and at the Hausdorff Center for Mathematics, Bonn. 

\section{Recollections and notation}\label{2}
\subsection{Quivers}Let $Q$ be a quiver with vertices $Q_0$ and arrows $Q_1$ denoted by $\alpha:i\rightarrow j$. We denote by $\Lambda=\Z Q_0$ the free abelian group over $Q_0$ and by $\Lambda^+=\N Q_0$ the set of dimension vectors written as $d=\sum_{i\in Q_0}d_ii$. Define $'\Lambda^+:=\Lambda^+\backslash\{0\}$. There exists a bilinear form on $\Lambda$, called the Euler form, given by
\[\langle d,e\rangle=\sum_{i\in Q_0}d_ie_i-\sum_{\alpha:i\rightarrow j}d_ie_j.\]
We denote its antisymmetrization by $\{d,e\}=\langle d,e\rangle-\langle e,d\rangle$. A representation $X$ of $Q$ of dimension $d\in\Lambda^+$ is given by complex vector spaces $X_i$ of dimension $d_i$ for every $i\in Q_0$ and by linear maps $X_{\alpha}:X_i\rightarrow X_j$ for every arrow $\alpha:i\rightarrow j\in Q_1$.
A vertex $q$ is called a sink (resp. source) if there does not exists an arrow starting at (resp. terminating at) $q$. A quiver is bipartite if $Q_0=I\cup J$ with sources $I$ and sinks $J$. In the following, we denote by $Q(I)\cup Q(J)$ the decomposition into sources and sinks.\\

For a fixed vertex $q\in Q_0$ we denote by $N_q$ the set of neighbours of $q$ and for $V\subset Q_0$ define $N_V:=\cup_{q\in V}N_q$.

For a representation $X$ of the quiver $Q$ we denote by $\underline{\dim} X\in\Lambda^+$ its dimension vector. Moreover we choose a level $l:Q_0\rightarrow\Nn^+$ on the set of vertices. Define two linear forms $\Theta, \kappa\in\mathrm{Hom}(\Zn Q_0,\Zn)$ by $\Theta(d)=\sum_{q\in Q_0}\Theta_qd_q$, $\kappa(d)=\sum_{q\in Q_0}l(q)d_q$ and a slope function $\mu:\Nn Q_0\rightarrow\Qn$ by \[\mu(d)=\frac{\Theta(d)}{\kappa(d)}.\]

For $\mu\in\Qn$ we denote by $'\Lambda^+_{\mu}\subset~ '\Lambda^+$ be the set of dimension vectors of slope $\mu$ and define $\Lambda_{\mu}^+=~'\Lambda^+_{\mu}\cup\{0\}$. This is a subsemigroup of $\Lambda^+$.

For a representation $X$ of the quiver $Q$ we define $\mu(X):=\mu(\underline{\dim}X)$. The representation $X$ is called (semi-)stable if the slope (weakly) decreases on proper non-zero subrepresentations. Fixing a slope function as above, we denote by $R^{\Theta-\rm{sst}}_d(Q)$ the set of semistable points and by $R^{\Theta-\mathrm{st}}_d(Q)$ the set of stable points in the affine variety $R_d(Q):=\oplus_{\alpha:i\rightarrow j}\mathrm{Hom}(\C^{d_i},\C^{d_j})$ of representations of dimension $d\in\Nn Q_0$.
There exist moduli spaces $M^{\Theta-\rm{st}}_d(Q)$ (resp. $M^{\Theta-\rm{sst}}_d(Q)$) of stable (resp. semistable) representations parametrizing isomorphism classes of stable (resp. polystable) representations (\cite{king}). If $Q$ is acyclic and $M^{\Theta-\rm{st}}_d(Q)$ is non-empty, it is a smooth irreducible variety of dimension $1-\langle d,d\rangle$. Moreover it is projective if semistability and stability coincide.\\

Fixing a quiver $Q$ and a dimension vector $d\in\mathbb{N}Q_0$ such that there exists a (semi-)stable representation for this tuple we call this tuple (semi-)stable.
\subsection{The tropical vertex}
We briefly review the definition of one of the tools we shall use, the tropical vertex group, following \cite[Section 0]{gps}.\\[1ex]
We fix nonnegative integers $l_1,l_2\geq 1$ and define $R$ as the formal power series ring $R=\Qn [[s_1,\ldots,s_{l_1},t_1,\ldots,t_{l_2}]]$, with maximal ideal $\mathfrak{m}$. Let $B$ be the $R$-algebra $$B=\Qn[x^{\pm 1},y^{\pm 1}][[s_1,\ldots,s_{l_1},t_1,\ldots,t_{l_2}]]=\Qn[x^{\pm 1},y^{\pm 1}]\widehat{\otimes}R$$
(a suitable completion of the tensor product).  For $(a,b)\in\Zn^2$ and a series $$f\in 1+x^ay^b\Qn[x^ay^b]\widehat{\otimes}\mathfrak{m}$$ we consider the $R$-linear automorphism of $B$ defined by
$$\theta_{(a,b),f}:\left\{\begin{array}{ccc}x&\mapsto&xf^{-b}\\ y&\mapsto&yf^a.\end{array}\right.$$
Notice that these automorphisms respect the symplectic form $\frac{dx}{x}\wedge\frac{dy}{y}$.
\begin{defi} The tropical vertex group $\mathbb{V}_R \subset{\rm Aut}_R(B)$ is defined as the completion with respect to $\mathfrak{m}$ of the subgroup of ${\rm Aut}_R(B)$ generated by all elements $\theta_{(a,b),f}$ as above.
\end{defi}

We recall that by \cite{KSAffine} (see also \cite[Theorem 1.3]{gps}) there exists a unique infinite ordered product factorization in $\mathbb{V}_R$ of the form 
\begin{equation*} \theta_{(1,0),\prod_k(1+s_kx)}\theta_{(0,1),\prod_l(1+t_ly)}=\prod_{b/a\mbox{ \footnotesize decreasing}}\theta_{(a,b),f_{(a,b)}},\end{equation*}
the product ranging over all coprime pairs $(a,b)\in\Nn^2$.
\section{Motivic MPS formula}

\subsection{Some symmetric function identities}
We start with some preliminaries on symmetric functions following \cite{McD}. Partitions $\lambda$ of $n$ are written $\lambda\vdash n$. With a partition $\lambda\vdash n$ we associate a multiplicity vector $m_*(\lambda)$, where $m_i(\lambda)$ is the multiplicity of the part $i$ in $\lambda$. Conversely with a vector $m_*=(m_l\in{\mathbb N})_{l\geq 1}$ we associate the partition $\lambda(m_*)=(1^{m_1}2^{m_2}\ldots)$. This induces a bijection between partitions of $n$ and the set of multiplicity vectors $m_*$ such that $\sum_llm_l=n$. In this case, we also write $m_*\vdash n$.\\[1ex]
Denote by $\mathcal{P}_{\mathbb Q}$ the ring of symmetric functions with rational coefficients in variables $x_i$ for $i\geq 1$.
We consider the so-called principal specialization map $p:\mathcal{P}_{\mathbb Q}\rightarrow {\mathbb Q}(q)$ given by $x_i\mapsto q^{i-1}$ for all $i\geq 1$.\\[1ex]
Denote by $e_n=\sum_{i_1<\ldots<i_n}x_{i_1}\ldots x_{i_n}$ the $n$-th elementary symmetric function and by $p_n=\sum_ix_i^n$ the $n$-th power sum function. We consider the following generating functions in $\mathcal{P}_{\mathbb Q}[[t]]$:
$$E(t)=\sum_{n\geq 0}e_nt^n,\;\;\; P(t)=\sum_{n\geq 1}p_nt^{n-1}.$$
Then $P(-t)=(\frac{d}{dt}E(t))/E(t)$ by \cite[I,(2.10')]{McD}. Defining $e_\lambda=e_{\lambda_1}e_{\lambda_2}\ldots$, and $p_\lambda$ similarly, for an arbitrary partition $\lambda$, both sets $(e_\lambda)_\lambda$, resp. $(p_\lambda)_\lambda$, are bases of $\mathcal{P}_{\mathbb Q}$ (\cite[I,(2.4), (2.12)]{McD}). We consider the base change between these bases.\\[1ex]
We have
$$e_n=\sum_{\lambda\vdash n}\varepsilon_\lambda z_\lambda^{-1}p_\lambda,$$
where $\varepsilon_\lambda=(-1)^{|\lambda|-l(\lambda)}$ and $z_\lambda=\prod_lm_l(\lambda)!l^{m_l(\lambda)}$ (\cite[I,(2.14')]{McD}).
\newpage
\begin{lem}\label{l1}$~$
\begin{enumerate}
\item The previous identity can be rewritten as
$$e_n=\sum_{m_*\vdash n}\prod_l\frac{1}{m_l!}\left(\frac{(-1)^{l-1}}{l}\right)^{m_l}p_{\lambda(m_*)}.$$
\item Conversely, we have:
$$p_n=(-1)^{n-1}n\sum_{\lambda\vdash n}\frac{(-1)^{l(\lambda)-1}}{l(\lambda)}\frac{l(\lambda)!}{\prod_lm_l(\lambda)!}e_\lambda.$$
\end{enumerate}
\end{lem}

\begin{proof} The first identity follows from the definitions. For the second we use $\log(1+x)=\sum_{k\geq 1}\frac{(-1)^{k-1}}{k}x^k$ to calculate
\begin{eqnarray*}\log E(t)&=&\log(1+\sum_{n\geq 1}e_nt^n)=\sum_{k\geq 1}\frac{(-1)^{k-1}}{k}(\sum_{n\geq 1}e_nt^n)^k=\\
&=&\sum_{k\geq 1}\frac{(-1)^{k-1}}{k}\sum_{n_1,\ldots,n_k\geq 1}e_{n_1}\ldots e_{n_k}t^{n_1+\ldots+n_k}=\\
&=&\sum_{k\geq 1}\frac{(-1)^{k-1}}{k}\sum_{\lambda:l(\lambda)=k}\frac{k!}{\prod_{l}m_l(\lambda)!}e_\lambda t^{|\lambda|}=\\
&=&\sum_{n\geq 1}\sum_{\lambda\vdash n}\frac{(-1)^{l(\lambda)-1}}{l(\lambda)}\frac{l(\lambda)!}{\prod_lm_l(\lambda)!}e_\lambda t^n,
\end{eqnarray*}
where we have used the fact that the number of rearrangements of a partition $\lambda$ is $\frac{l(\lambda)!}{\prod_lm_l(\lambda)!}$.
Differentiating and using $P(-t)=\frac{d}{dt}(\log E(t))$, the lemma follows.\end{proof}

For partitions $\lambda$, $\mu$ of $n$, denote by $L_{\mu\lambda}$ the number of functions $f:\{1,\ldots,l(\mu)\}\rightarrow{\mathbb N}$ such that $\lambda_i=\sum_{j:f(j)=i}\mu_j$. Then $e_\lambda=\sum_\mu\varepsilon_\mu z_\mu^{-1}L_{\mu\lambda}p_\mu$ by \cite[I,(6.11)]{McD}.

\begin{lem}\label{l2} We have
$$e_\lambda=\sum_{m_*\vdash n}(\sum_{m^*_*}\prod_{l\geq 1}\frac{m_l!}{\prod_j m^j_l!})\prod_{l\geq 1}\frac{1}{m_l!}\left(\frac{(-1)^{l-1}}{l}\right)^{m_l}p_{\lambda(m_*)},$$
where the inner sum ranges over all tuples $(m^j_l)_{j,l\geq 1}$ such that $\sum_jm^j_l=m_l$ for all $l\geq 1$ and $\sum_llm^j_l=\lambda_j$ for all $j\geq 1$.
\end{lem}

\begin{proof} Assume $\mu=(1^{m_1}2^{m_2}\ldots)$. To a function $f$ as above, we associate the sets
$$I^j_l=\{i\in\{1,\ldots,m_l\}\, :\, f(m_1+\ldots+m_{l-1}+i)=j\}$$
for $j,l\geq 1$. Then, for all $l\geq 1$, the set $\{1,\ldots,m_l\}$ is the disjoint union of the $I^j_l$. Defining $m^j_l$ as the cardinality of $I^j_l$, we thus have $\sum_jm^j_l=m_l$ for all $l\geq 1$, and $\sum_llm^j_l=\lambda j$ by definition of $f$. This establishes a bijection between the set of functions $f$ which is counted by $L_{\mu\lambda}$ and the set of pairs $(m_*^*,I^*_*)$ which is counted by the inner sum. \end{proof}

Under the principal specialization map, $e_n$ maps to $\frac{q^{\binom{n}{2}}}{(1-q)(1-q^2)\ldots(1-q^n)}$, and $p_n$ maps to $\frac{1}{(1-q^n)}$ (\cite[I,2. Example 4.]{McD}). Using the first identity of Lemma \ref{l1}, this yields the identity
$$\frac{q^{\binom{n}{2}}}{(1-q)\ldots(1-q^n)}=\sum_{m_*\vdash n}\prod_{l\geq 1}\frac{1}{m_l!}\left(\frac{(-1)^{l-1}}{l(1-q^l)}\right)^{m_l}.$$
Replacing $q$ by $q^{-1}$ and multiplying by an appropriate power of $q$, this is equivalent to
\begin{lem}\label{l3} We have
$$\frac{q^{\binom{n}{2}}}{(q^n-1)\ldots(q^n-q^{n-1})}=\sum_{m_*\vdash n}\prod_{l\geq 1}\frac{1}{m_l!}\left(\frac{(-1)^{l-1}}{l[l]_q}\right)^{m_l}\frac{1}{(q-1)^{\sum_lm_l}},$$
where $[l]_q=(q^l-1)/(q-1)$.
\end{lem}
 
\subsection{The motivic MPS formula}


We assume throughout that an arbitrary finite quiver $Q$ and a stability $\Theta$ for $Q$ are given and we fix a vertex $i\in Q_0$.\\[1ex]
We introduce a new (levelled) quiver $\widehat{Q}$ by replacing the vertex $i$ by vertices $i_{k,l}$ for $k,l\geq 1$, thus $\widehat{Q}_0=Q_0\setminus\{i\}\cup\{i_{k,l}\, :\, k,l\geq 1\}$, with vertex $i_{k,l}$ being of level $l$. The arrows in $\widehat{Q}$ are given by the following rules:
\begin{itemize}
\item all arrows $\alpha:j\rightarrow k$ in $Q$ which are not incident with $i$ induce an arrow $\alpha:j\rightarrow k$ in $\widehat{Q}$,
\item all arrows $\alpha:i\rightarrow j$ (resp.~$\alpha:j\rightarrow i$) in $Q$ for $j\not=i$ induce arrows $\alpha_p:i_{k,l}\rightarrow j$ (resp. $\alpha_p:j\rightarrow i_{k,l}$) for $k,l\geq 1$ and $p=1,\ldots,l$ in $\widehat{Q}$,
\item all loops $\alpha:i\rightarrow i$ in $Q$ induce arrows $\alpha_{p,q}:i_{k,l}\rightarrow i_{k',l'}$ for $k,l,k',l'\geq 1$ and $p=1,\ldots,l$, $q=1,\ldots,l'$ in $\widehat{Q}$.
\end{itemize}

Given a dimension vector $d$ for $Q$ and a multiplicity vector $m_*\vdash d_i$ (that is, $\sum_llm_l=d_i$) as above, we define a dimension vector $\widehat{d}(m_*)$ for $\widehat{Q}$ by $\widehat{d}(m_*)_j=d_j$ for $j\not=i$ in $Q_0$ and $$\widehat{d}(m_*)_{i_{k,l}}=\left\{\begin{array}{rrr}1&,&k\leq m_l\\ 0&,&k>m_l.\end{array}\right.$$

We have $$G_{\widehat{d}(m_*)}\simeq\prod_{j\not=i}{\rm GL}_{d_j}({\mathbb C})\times{\bf G}_m^{\sum_lm_l},$$
where ${\bf G}_m={\mathbb C}^*$ denotes the multiplicative group of the field ${\mathbb C}$.
We choose an arbitrary basis of ${\mathbb C}^{d_i}$ indexed by vectors $v_{k,l,p}$ for $l\geq 1$, $1\leq k\leq m_l$ and $1\leq p\leq l$ (this is possible since $\sum_llm_l=d_i$). Then the group $G_{\widehat{d}(m_*)}$ embeds into $G_d$ by letting the $(k,l)$-th component of ${\bf G}_m$ scale the vectors $v_{k,l,p}$ for $p=1,\ldots,l$ simultaneously, for all $l\geq 1$, $1\leq k\leq m_l$.\\[1ex] 
We define a stability $\widehat{\Theta}$ for $\widehat{Q}$ by $\widehat{\Theta}_j=\Theta_j$ for all $j\not=i$ in $Q_0$ and $\widehat{\Theta}_{i_{k,l}}=l\Theta_i$ for all $k,l\geq 1$. The associated slope function is denoted by $\widehat{\mu}$. The following lemma is easily verified by working through the definitions of $\widehat{Q}$, $\widehat{d}$ and $\widehat{\mu}$:

\begin{lem}\label{l34} Via the above embedding, we have a $G_{\widehat{d}(m_*)}$-equivariant isomorphism between $R_d(Q)$ and $R_{\widehat{d}(m_*)}(\widehat{Q})$. Furthermore, we have $\widehat{\mu}(\widehat{d}(m_*))=\mu(d)$.
\end{lem}

\begin{proof} \end{proof}

Our motivic version of the MPS formula is an identity in a suitably localized Grothendieck ring of varieties; we refer to \cite{Bridge} for an introduction to this topic suitable for our purposes. Let $K_0({\rm Var}/{\mathbb C})$ be the free abelian group generated by representatives $[X]$ of all isomorphism classes of complex varieties $X$, modulo the relation $[X]=[A]+[U]$ if $A$ is isomorphic to a closed subvariety of $X$, with complement isomorphic to $U$. Multiplication is given by $[X]\cdot[Y]=[X\times Y]$. Denote by $[{\bf L}]$ the class of the affine line. We work in the localization
$$\mathcal{K}=(K_0({\rm Var}/{\mathbb C})\otimes{\mathbb Q})[[{\bf L}]^{-1},([{\bf L}]^n-1)^{-1}\, :\, n\geq 1].$$

\begin{thm}\label{motivicMPS} For arbitrary $Q$, $d$, $\Theta$ and $i$ as above, the following identity holds in $\mathcal{K}$:
$$[{\bf L}]^{\binom{d_i}{2}}\frac{[R_d^{\rm sst}(Q)]}{[G_d]}=\sum_{m_*\vdash d_i}\prod_{l\geq 1}\frac{1}{m_l!}\left(\frac{(-1)^{l-1}}{l[{\mathbb P}^{l-1}]}\right)^{m_l}\frac{[R_{\widehat{d}(m_*)}^{\rm sst}(\widehat{Q})]}{[G_{\widehat{d}(m_*)}]}.$$
\end{thm}

\begin{proof} We start the proof by translating the identity of Lemma \ref{l3} into the ring $\mathcal{K}$. We note the following identities:
$$[{\rm GL}_n({\mathbb C})]=\prod_{i=0}^{n-1}([{\bf L}]^n-[{\bf L}]^i),\;\;\; [{\bf G}_m]=[{\bf L}]-1,\;\;\; [{\mathbb P}^{n-1}]=([{\bf L}]^n-1)/([{\bf L}]-1)=[n]_{[\bf L]}.$$

Then the above identity translates into
$$\frac{[{\bf L}]^{\binom{n}{2}}}{[{\rm GL}_n({\mathbb C})]}=\sum_{m_*\vdash n}\prod_{l\geq 1}\frac{1}{m_l!}\left(\frac{(-1)^{l-1}}{l[{\mathbb P}^{l-1}]}\right)^{m_l}\frac{1}{[{\bf G}_m^{\sum_lm_l}]}.$$

Replacing $n$ by $d_i$, multiplying by $[R_d(Q)]/\prod_{j\not=i}[{\rm GL}_{d_j}({\mathbb C})]$ and using the above identifications and Lemma \ref{l34}, this yields the MPS formula for trivial stability $\Theta=0$:
$$[{\bf L}]^{\binom{d_i}{2}}\frac{[R_d(Q)]}{[G_d]}=\sum_{m_*\vdash d_i}\prod_{l\geq 1}\frac{1}{m_l!}\left(\frac{(-1)^{l-1}}{l[{\mathbb P}^{l-1}]}\right)^{m_l}\frac{[R_{\widehat{d}(m_*)}(\widehat{Q})]}{[G_{\widehat{d}(m_*)}]}.$$

Now we make use of the Harder-Narasimhan stratification of $R_d(Q)$ constructed in \cite{ReHNS}: we fix a decomposition $d=d^1+\ldots+d^s$ into non-zero dimension vectors such that $\mu(d^1)>\ldots>\mu(d^s)$, which we call a HN type for $d$, denoted by $d^*=(d^1,\ldots,d^s)\models d$. Denote by $R_d^{d^*}(Q)$ the set of all representations $M\in R_d(Q)$ such that in the Harder-Narasimhan filtration $0=M^0\subset M^1\subset\ldots\subset M^s=M$ of $M$, the dimension vector of $M^i/M^{i-1}$ equals $d^i$ for all $i=1,\ldots,s$. By \cite{ReHNS}, we have
$$R_d^{d^*}(Q)\simeq G_d\times^{P_{d^*}}V_{d^*},$$
where $P_{d^*}$ is a parabolic subgroup of $G_d$ with Levi isomorphic to $\prod_{k=1}^s G_{d^k}$, and $V_{d^*}$ is a vector bundle over $\prod_{k=1}^s R_{d^k}^{\rm sst}(Q)$ of rank $r_{d^*}=\sum_{k<l}\sum_{p\rightarrow q}d^l_pd^k_q$. This implies the following identity in $\mathcal{K}$:
$$[R_d(Q)]=\sum_{d^*\models d}[G_d\times^{P_{d^*}}V_{d^*}]=\sum_{d^*\models d}\frac{[G_d]}{[P_{d^*}]}[{\bf L}]^{r_{d^*}}[\prod_k R_{d^k}^{\rm sst}(Q)].$$
Using $[P_{d^*}]=[{\bf L}]^{\sum_{k<l}\sum_id^l_id^k_i}\prod_k[G_{d^k}]$ and dividing by $[G_d]$, this yields a motivic version of the Harder-Narasimhan recursion of \cite{ReHNS} determining $[R_d^{\rm sst}(Q)]/[G_d]$ recursively:
$$\frac{[R_d(Q)]}{[G_d]}=\sum_{d^*\models d}[{\bf L}]^{-\sum_{k<l}\langle d^l,d^k\rangle}\prod_{k=1}^s\frac{[R_{d^k}^{\rm sst}(Q)]}{[G_{d^k}]}.$$

We now derive the MPS formula by induction over the dimension vector $d$. The induction starts with $d$ of total dimension one. In this case, all points are semistable, and the formula is already proved. We now compute $[{\bf L}]^{\binom{d_i}{2}}[R_d(Q)]/[G_d]$ in two ways. By applying first the MPS formula for trivial stability, then the HN recursion, we get
$$[{\bf L}]^{\binom{d_i}{2}}\frac{[R_d(Q)]}{[G_d]}=\sum_{m_*\vdash d_i}\prod_{l\geq 1}\frac{1}{m_l!}\left(\frac{(-1)^{l-1}}{l[{\mathbb P}^{l-1}]}\right)^{m_l}\sum_{\widehat{d}^*\models\widehat{d}(m_*)}[{\bf L}]^{-\sum_{k<l}\langle \widehat{d}^l,\widehat{d}^k\rangle}\prod_{k=1}^s\frac{[R_{\widehat{d}^k}^{\rm sst}(\widehat{Q})]}{[G_{\widehat{d}^k}]}.$$
In choosing a HN type $\widehat{d}^*\models \widehat{d}(m_*)$, we choose a HN type $d^*\models d$, together with set partitions $\{1,\ldots,m_l\}=\bigcup_kI^k_l$ for all $l\geq 1$ such that for $m^k_l=|I^k_l|$, we have $\sum_llm^k_l=d^k_i$ for all $k=1,\ldots,s$; Lemma \ref{l34} ensures that the respective slope conditions are compatible. Note that $R_{\widehat{d}^k}^{\rm sst}(\widehat{Q})$ and $G_{\widehat{d}^k}$ only depend on $d^*$ and on the $m^k_l$, but not on the actual parts $I^k_l$; in fact, $\widehat{d}^k=\widehat{d^k}(m^k_*)$. A short calculation using the definition of $\widehat{Q}$ shows that
$$\langle d^l,d^k\rangle-\langle \widehat{d}^l,\widehat{d}^k\rangle=d^l_id^k_i.$$
Thus the above sum can be rewritten as

$$\sum_{d^*\models d}\sum_{(m^k_*\vdash d^k_i)_k}\prod_{l\geq 1}\frac{(\sum_km^k_l)!}{\prod_km^k_l!}\prod_{l\geq 1}\frac{1}{\sum_km^k_l)!}\left(\frac{(-1)^{l-1}}{l[{\mathbb P}^{l-1}]}\right)^{\sum_km^k_l}\times$$
$$\times[{\bf L}]^{-\sum_{k<l}\langle d^l,d^k\rangle+\sum_{k<l}d^l_id^k_i}\prod_{k=1}^s\frac{[R_{\widehat{d^k}(m^k_*)}^{\rm sst}(\widehat{Q})]}{[G_{\widehat{d^k}(m^k_*)}]}=$$

$$=[{\bf L}]^{\binom{d_i}{2}}\sum_{d^*\models d}[{\bf L}]^{-\sum_{k<l}\langle d^l,d^k\rangle}\prod_{k=1}^s[{\bf L}]^{-{\binom{d^k_i}{2}}}\sum_{m_*\vdash d^k_i}\prod_{l\geq 1}\frac{1}{m_l!}\left(\frac{(-1)^{l-1}}{l[{\mathbb P}^{l-1}]}\right)^{m_l}\frac{[R_{\widehat{d^k}(m^k_*)}^{\rm sst}(\widehat{Q})]}{[G_{\widehat{d^k}(m^k_*)}]}.$$
On the other hand, $[{\bf L}]^{\binom{d_i}{2}}[R_d(Q)]/[G_d]$ evaluates to
$$[{\bf L}]^{\binom{d_i}{2}}\frac{[R_d(Q)]}{[G_d]}=[{\bf L}]^{\binom{d_i}{2}}\sum_{d^*\models d}[{\bf L}]^{-\sum_{k<l}\langle d^l,d^k\rangle}\prod_{k=1}^s\frac{[R_{d^k}^{\rm sst}(Q)]}{[G_{d^k}]}$$
by the HN recursion. Using the inductive hypothesis, all summands of the last two sums corresponding to HN types of length at least two coincide, thus the summands corresponding to the trivial HN type $d$ coincide; this yields the MPS formula for $d$.\end{proof}

Sometimes it might be convenient to rewrite the MPS formula in terms of partitions instead of multiplicity vectors:

\begin{kor} We have
$$[{\bf L}]^{\binom{d_i}{2}}\frac{[R_d^{\rm sst}(Q)]}{[G_d]}=\sum_{\lambda\vdash d_i}\varepsilon_\lambda z_\lambda^{-1}\frac{1}{\prod_j[{\mathbb P}^{\lambda_j-1}]}\frac{[R_{\widehat{d}(m_*(\lambda))}^{\rm sst}(\widehat{Q})]}{[G_{\widehat{d}(m_*(\lambda))}]}.$$
\end{kor}

There is a well-defined ring homomorphism $\pi:\mathcal{K}\rightarrow {\mathbb Q}(t)$ mapping the class of a smooth projective variety $X$ to its Poincar\'e polynomial $P(X,t)=\sum_i\dim H^i(X,{\mathbb Q})t^i$ in singular cohomology. In the case where the dimension vector $d$ is $\Theta$-coprime, that is, $\mu(e)\not=\mu(d)$ for all non-zero dimension vectors $e<d$, the moduli space $M^{\rm sst}_d(Q)=R_d^{\rm sst}(Q)/G_d({\mathbb C})$ is a smooth variety, and we have
$$P(M^{\rm sst}_d(Q),t)=(t^2-1)\cdot \pi([R^{\rm sst}_d(Q)]/[G_d])$$
by \cite[Theorem 2.5]{ER}. Specialization of the motivic MPS formula to this case yields a formula for its Poincar\'e polynomial, and in particular for its Euler characteristic:

\begin{kor} If $d$ is $\Theta$-coprime, we have
$$t^{d_i(d_i-1)}P(M^{\rm sst}_d(Q),t)=\sum_{m_*\vdash d_i}\prod_{l\geq 1}\frac{1}{m_l!}\left(\frac{(-1)^{l-1}}{l[l]_{t^2}}\right)^{m_l}P(M_{\widehat{d}(m_*)}^{\rm sst}(\widehat{Q}),t)$$
and
$$\chi(M^{\rm sst}_d(Q))=\sum_{m_*\vdash d_i}\prod_{l\geq 1}\frac{1}{m_l!}\left(\frac{(-1)^{l-1}}{l^2}\right)^{m_l}\chi(M_{\widehat{d}(m_*)}^{\rm sst}(\widehat{Q})).$$
\end{kor}

\subsection{Dual MPS formula}
We define a quiver $\check{Q}$ as the ``level one part" of $\widehat{Q}$, thus $\check{Q}_0=Q_0\setminus\{i\}\cup\{i_{k}\, :\, k\geq 1\}$, and the arrows in $\check{Q}$ are given by the following rules:
\begin{itemize}
\item all arrows $\alpha:j\rightarrow k$ in $Q$ which are not incident with $i$ induce an arrow $\alpha:j\rightarrow k$ in $\check{Q}$,
\item all arrows $\alpha:i\rightarrow j$ (resp.~$\alpha:j\rightarrow i$) in $Q$ for $j\not=i$ induce arrows $\alpha:i_{k}\rightarrow j$ (resp. $\alpha:j\rightarrow i_{k}$) for $k\geq 1$ in $\check{Q}$,
\item all loops $\alpha:i\rightarrow i$ in $Q$ induce arrows $\alpha_:i_{k}\rightarrow i_{k'}$ for $k,k'\geq 1$ in $\check{Q}$.
\end{itemize}

Given a dimension vector $d$ for $Q$ and a partition $\lambda\vdash d_i$, we define a dimension vector $\check{d}(\lambda)$ for $\check{Q}$ by $\check{d}(\lambda)_j=d_j$ for $j\not=i$ in $Q_0$ and $\check{d}(\lambda)_{i_{k}}=\lambda_k$ for $k\geq 1$.\\[1ex]
Application of the MPS formula to all vertices $i_k$ for $k\geq 1$ of $\check{Q}$ yields:
$$[{\bf L}]^{\sum_j{\binom{\lambda_j}{2}}}\frac{[R_{\check{d}(\lambda)}^{\rm sst}(\check{Q})]}{[G_{\check{d}(\lambda)}]}=\sum_{(m^j_*\vdash\lambda_j)_j}\prod_{j\geq 1}\prod_{l\geq 1}\frac{1}{m^j_l!}\left(\frac{(-1)^{l-1}}{l[{\mathbb P}^{l-1}]}\right)^{m_l^j}\frac{[R_{\widehat{d}(\sum_jm^j_*)}^{\rm sst}(\widehat{Q})]}{[G_{\widehat{d}(\sum_jm^j_*)}]}=$$
$$=\sum_{m_*\vdash d_i}(\sum_{m^*_*}\prod_{l\geq 1}\frac{m_l!}{\prod_jm^j_l!})\prod_{l\geq 1}\frac{1}{m_l!}\left(\frac{(-1)^{l-1}}{l[{\mathbb P}^{l-1}]}\right)^{m_l}\frac{[R_{\widehat{d}(m_*)}^{\rm sst}(\widehat{Q})]}{[G_{\widehat{d}(m_*)}]},$$
where the inner sum runs over all tuples $(m^j_l)_{j,l\geq 1}$ such that $m_l=\sum_jm^j_l$ for all $l$ and $\lambda_j=\sum_llm^j_l$ for all $j$. Comparison with Lemma \ref{l2} yields:

\begin{pro} There is a well-defined map $\mathcal{P}_{\mathbb Q}\rightarrow \mathcal{K}$ of ${\mathbb Q}$-vector spaces such that
$$e_\lambda\mapsto [{\bf L}]^{\sum_j{\binom{\lambda_j}{2}}}\frac{[R_{\check{d}(\lambda)}^{\rm sst}(\check{Q})]}{[G_{\check{d}(\lambda)}]},\;\;\; p_\lambda\mapsto \prod_j\frac{1}{[{\mathbb P}^{\lambda_j-1}]}\cdot\frac{[R_{\widehat{d}(m_*(\lambda))}^{\rm sst}(\widehat{Q})]}{[G_{\widehat{d}(m_*(\lambda))}]}.$$
\end{pro}

It might be interesting to ask whether the images of other bases of the ring of symmetric functions (monomial symmetric functions, complete symmetric functions, Schur functions, ...) have a natural interpretation in terms of motives of quiver moduli.\\[2ex]
We can now map the second identity of Lemma \ref{l1} to $\mathcal{K}$ to get the following dual version of the MPS formula:

\begin{kor}\label{dualMPS} For given $d$, denote by $\widehat{d}^0$ the dimension vector for $\widehat{Q}$ with a single entry $1$ on level $d_i$. Then
$$\frac{1}{[{\mathbb P}^{d_i-1}]}\cdot\frac{[R_{\widehat{d}^0}^{\rm sst}(\widehat{Q})]}{[G_{\widehat{d}^0}(\widehat{Q})]}=(-1)^{d_i-1}d_i\sum_{\lambda\vdash d_i}(-1)^{l(\lambda)-1}\frac{(l(\lambda)-1)!}{\prod_lm_l(\lambda)!}[{\bf L}]^{\sum_j{\binom{\lambda_j}{2}}}\frac{[R_{\check{d}(\lambda)}^{\rm sst}(\check{Q})]}{[G_{\check{d}(\lambda)}]}.$$
\end{kor}

\section{The MPS formula as a degeneration formula in Gromov-Witten theory}\label{4}
In the rest of the paper for every bipartite quiver $Q$ we consider the linear form $\Theta\in\mathbb{N}Q_0$ defined by $\Theta_i=1$ for every $i\in Q(I)$ and $\Theta_j=0$ for every $j\in Q(J)$. Additionally fixing a level $l:Q\rightarrow\mathbb{N}^+$, we define the linear form $\Theta_l$ by $(\Theta_l)_q=l(q)\Theta_q$ and consider the slope $\mu=\Theta_l/\kappa$ where $\kappa$ is defined as in Section \ref{2}. Note that for the trivial level structure, i.e. $l(q)=1$ for every $q\in Q_0$, we have $\Theta_l=\Theta$ and $\kappa=\dim$.

In this section we specialize the MPS formula to Euler characteristics, and at the same time we restrict to a special class of quivers. These are the complete bipartite quivers $K(l_1, l_2)$ of \cite{rw} Section 5, defined by the vertices 
\begin{equation*}
K(l_1,l_2)_0= \{i_1,\ldots,i_{l_1}\}\cup\{j_1,\ldots,j_{l_2}\}
\end{equation*}
and the arrows 
\begin{equation*}
K(l_1,l_2)_1=\{\alpha_{k,l}:i_k\rightarrow j_l\mid k\in\{1,\ldots l_1\},l\in\{1,\ldots l_2\}\}.
\end{equation*}
A dimension vector for $K(l_1, l_2)$ is uniquely determined by a pair of \emph{ordered} partitions 
\begin{equation*}
({\bf P}_1,{\bf P}_2) = (\sum^{l_1}_{i=1} p_{1 i}, \sum^{l_2}_{j=1} p_{2 j}).
\end{equation*} 
We assume throughout this section that the sizes $|{\bf P}_1|$, $|{\bf P}_2|$ are \emph{coprime}. We fix the trivial level structure given by $l(q)=1$ for all $q\in K(l_1,l_2)_0$. We denote by 
\begin{equation*}
M^{\Theta-\rm st}({\bf P}_1,{\bf P}_2) = M^{\Theta-\rm st}_{({\bf P}_1,{\bf P}_2)}(K(l_1,l_2))
\end{equation*} 
the moduli space of stable representations with respect to this choice. 

The MPS formula in this context can be expressed uniformly for all $l_1, l_2$ and all dimension vectors by introducing an infinite quiver $\mathcal{N}$ with a suitable level structure.  We define its vertices by
\begin{equation*}
\mathcal{N}_0=\{i_{(w,m)}\mid (w,m)\in\Nn^2\}\cup \{j_{(w,m)}\mid (w,m)\in\Nn^2\},
\end{equation*}
and the arrows by 
\begin{equation*}
\mathcal{N}_1=\{\alpha_1,\ldots,\alpha_{w\cdot w'}:i_{(w,m)}\rightarrow j_{(w',m')},\forall\,w,w',m,m'\in\Nn\}.
\end{equation*}
The level function is given by
\begin{equation*}
l(q_{(w,m)})=w,\forall\,q\in\{i,j\}, m\in\Nn
\end{equation*}
and we fix the linear form $\Theta_l$.
A \emph{refinement} of $({\bf P}_1, {\bf P}_2)$ is a pair of sets of integers 
\begin{equation*}
(k^1, k^2) = (\{k^1_{w i}\}, \{k^2_{w j}\})
\end{equation*} 
such that for $i=1,\ldots,l_1$ and $j=1,\ldots, l_2$ we have
\begin{equation*}
p_{1 i}=\sum_w wk_{w i}^1,\,p_{2 j}=\sum_w wk_{w j}^2.
\end{equation*}
We will denote refinements by $(k^1,k^2)\vdash ({\bf P}_1,{\bf P}_2)$. The number of entries of weight $w$ in $k^i$ is defined by
\begin{equation*}
m_w(k^i)=\sum_{j=1}^{l_i} k^i_{w j}.
\end{equation*}
A fixed refinement $(k^1, k^2)$ induces a dimension vector for $\mathcal{N}$ by setting 
\begin{equation*}
d_{q_{(w, m)}} = \left\{
\begin{matrix}
1\textrm{ for } m=1, \ldots, m_w(k^p),\\
0\textrm{ for } m > m_w(k^p),\,\,\,\,\,\,\,\,\,\,\,\,\,\,\,\,
\end{matrix}\right.
\end{equation*}
for $q \in \{i, j\}$, and $p = 1, 2$ for $q = i, j$. With this notation in place, the MPS formula at the level of Euler characteristics can be expressed by
\begin{equation}\label{mps}
\chi(M^{\Theta-\rm st}({\bf P}_1,{\bf P}_2)) = \sum_{(k^1,k^2)\vdash({\bf P}_1,{\bf P}_2)}\chi(M^{\Theta_l-\rm st}_{(k^1,k^2)}(\mathcal{N}))\prod^2_{i=1}\prod^{l_i}_{j=1}\prod_{w}\frac{(-1)^{k^i_{w, j}(w-1)}}{k^i_{w, j}!w^{2k^i_{w, j}}}.
\end{equation}

The ordered partition $({\bf P}_1,{\bf P}_2)$ also encodes an a priori very different kind of data, namely the Gromov-Witten invariant $N[({\bf P}_1,{\bf P}_2)]$ of \cite{gps} Section 0.4. Roughly speaking this is a virtual count of rational curves in the weighted projective plane $\PP(|{\bf P}_1|, |{\bf P}_2|, 1)$ which pass through $l_j$ specified distinct points lying on the distinguished toric divisor $D_j$ for $j = 1, 2$. We require that these points are not fixed by the torus action, that the multiplicities at the points are specified by ${\bf P}_j$, and that the curve touches the remaining toric divisor $D_{\rm out}$ at some point which is also not fixed by the torus. The refined GW/Kronecker correspondence of \cite{rw} (based on \cite{gps}, \cite{gp}) leads to a rather striking consequence (\cite{rw} Corollary 9.1):
\begin{equation*}\label{mpsfor}
N[({\bf P}_1,{\bf P}_2)] = \chi(M^{\Theta-\rm st}({\bf P}_1,{\bf P}_2)).
\end{equation*}

The powerful degeneration formula of Gromov-Witten theory (\cite{ionel}, \cite{ruan}, \cite{li}) allows one to express $N[({\bf P}_1,{\bf P}_2)]$ in terms of certain \emph{relative} Gromov-Witten invariants, enumerating rational curves with \emph{tangency} conditions, as we now briefly discuss. Following \cite{gps} Section 2.3, we define a \emph{weight vector} ${\bf w}_i$ as a sequence of integers $(w_{i1}, \ldots, w_{i t_i})$ with
\begin{equation*}
0 < w_{i1} \leq w_{i2} \leq \cdots \leq w_{i t_i}.
\end{equation*}
The \emph{automorphism group} $\Aut({\bf w}_i)$ of a weight vector ${\bf w}_i$ is the subset of the symmetric group on $t_i$ letters which stabilizes ${\bf w}_i$. A pair of weight vectors $({\bf w}_1, {\bf w}_2)$ encodes a relative Gromov-Witten invariant $N^{\rm rel}[({\bf w}_1, {\bf w}_2)]$, virtually enumerating rational curves in $\PP(|{\bf w}_1|,|{\bf w}_2|,1)$ which are tangent to $D_i$ at specified points (not fixed by the torus), with order of tangency specified by ${\bf w}_i$. The rigorous construction of these invariants is carried out in \cite{gps} Section 4.4. Let us now fix weight vectors ${\bf w}_i$ with $|{\bf w}_i| = |{\bf P}_i|$ for $i = 1, 2$. A \emph{set partition} $I_{\bullet}$ of ${\bf w}_i$ is a decomposition of the index set  
\begin{equation*}
I_1 \cup \cdots \cup I_{l_i} = \{1, \ldots, t_i\}
\end{equation*}
into $l_i$ disjoint, possibly empty parts. We say that the set partition $I_{\bullet}$ is \emph{compatible} with ${\bf P}_i$ (or simply compatible) if for all $j$ we have 
\begin{equation*}
p_{ij} = \sum_{r \in I_j} w_{ir}.
\end{equation*}
The relevant degeneration formula involves the ramification factors 
\begin{equation*}
R_{{\bf P}_i \mid {\bf w}_i} = \sum_{I_{\bullet}}\prod^{t_i}_{j = 1}\frac{(-1)^{w_{ij}-1}}{w^2_{ij}}, 
\end{equation*}
where we are summing over all compatible set partitions $I_{\bullet}$. Then Proposition 5.3 from \cite{gps} yields the equality
\begin{equation}\label{deg}
N[({\bf P}_1,{\bf P}_2)] = \sum_{({\bf w}_1, {\bf w}_2)} N^{\rm rel}[({\bf w}_1, {\bf w}_2)]\prod^2_{i = 1} \frac{\prod^{t_i}_{j = 1} w_{ij}}{|\Aut({\bf w}_i)|} R_{{\bf P}_i \mid {\bf w}_i}.
\end{equation}
We come to the central claim of this section:
\begin{thm}\label{mpsdeg} Given the equality $N[({\bf P}_1,{\bf P}_2)] = \chi(M^{\Theta-\rm st}({\bf P}_1,{\bf P}_2))$ (i.e. the coprime case of the refined GW/Kronecker correspondence), the MPS formula \eqref{mps} for the Euler characteristics of quiver representations is equivalent to the Gromov-Witten degeneration formula \eqref{deg}.
\end{thm}
The rest of this section is devoted to a proof of this result. As a first step, to simplify the comparison, we will rewrite the degeneration formula \eqref{deg} as a sum over pairs of refinements $(k^1, k^2)$ rather than pairs of weight vectors $({\bf w}_1, {\bf w}_2)$. Notice that a fixed refinement $k^i$ induces a weight vector ${\bf w}(k^i) = (w_{i1}, \ldots, w_{it_i})$ of length $t_i = \sum_w m_w(k^i)$, by
\begin{equation*}
w_{ij}=w\text{ for all }j=\sum_{r=1}^{w-1}m_r(k^i)+1,\ldots,\sum_{r=1}^{w}m_r(k^i).
\end{equation*}  
Of course the weight vector ${\bf w}(k^i)$ only depends on $k^i$ through $\{m_w(k^i)\}_w$. However we wish to think of ${\bf w}(k^i)$ as coming with a distinguished set partition $I_{\bullet}(k^i)$: the segment of weight $w$ entries in ${\bf w}(k^i)$ is partitioned into $l_i$ consecutive chunks of size $k^1_{w1}, \cdots, k^1_{w l_i}$, and we declare the indices for the $j$-th chunk to lie in $I_j(k^i)$. 
\begin{lem} The degeneration formula for Gromov-Witten invariants \eqref{deg} can be rewritten as
\begin{eqnarray*}N[({\bf P}_1,{\bf P}_2)]&=&\sum_{(k_1,k_2)\vdash ({\bf P}_1,{\bf P}_2)}N^{\rm rel}[({\bf w}(k^1), {\bf w}(k^2) )]\prod^2_{i=1}\prod^{l_i}_{j=1}\prod_{w}\frac{(-1)^{k^i_{w,j}(w-1)}}{k^i_{w,j}!w^{k^i_{w,j}}}.
\end{eqnarray*}
\end{lem}
\begin{proof} Consider the following operation on a compatible set partition $I_{\bullet}$: if there exist $a \in I_p$ and $b \in I_q$ with $p \neq q$ and $w_{i, a} = w_{i, b}$, then permuting the indices $p, q$ yields a new set partition $I'_{\bullet}$ which is still compatible. We write $[I_{\bullet}]$ for the equivalence class of set partitions generated by this operation. For a set partition $I_{\bullet}(k^i)$ which is induced by a refinement $k^i$, a simple count shows that the equivalence class $[I_{\bullet}(k^i)]$ contains
\begin{equation*}
\prod_w\binom{m_w(k^i)}{k^i_{w,1}}\binom{m_w(k^i)-k^i_{w,1}}{k^i_{w,2}}\ldots\binom{m_w(k^i)-\sum_{r=1}^{l_i-1}k^i_{w,r}}{k^i_{w,l_i}}
\end{equation*}
distinct elements. 

The formula \eqref{deg} is equivalent to
\begin{equation*}
N[({\bf P}_1,{\bf P}_2)] = \sum_{({\bf w}_1, {\bf w}_2)} \sum_{I^1_{\bullet}} \sum_{I^2_{\bullet}} N^{\rm rel}[({\bf w}_1, {\bf w}_2)] \prod^2_{i = 1} \frac{1}{|\Aut({\bf w}_i)|}\prod^{t_i}_{j = 1}\frac{(-1)^{w_{ij}-1}}{w_{ij}},  
\end{equation*}
where we are summing over all compatible set partitions. But we can enumerate the data ${\bf w}_i, I^i_{\bullet}$ differently: namely, rather than fixing ${\bf w}_i$ and considering all admissible $I^i_{\bullet}$, we can fix a refinement $k^i$, form the weight vector ${\bf w}(k^i)$ and restrict to the set partitions in the class $[I_{\bullet}(k^i)]$. In this case we have
\begin{equation*}
|\Aut({\bf w}(k^i))| = \prod_w m_w(k^i)!,\,\,\,\prod^{t_i}_{j = 1}\frac{(-1)^{w_{ij}-1}}{w_{ij}} = \prod^{l_i	}_{j=1}\prod_{w}\frac{(-1)^{k^i_{w,j}(w-1)}}{w^{k^i_{w,j}}}.
\end{equation*}
Thus the right hand side of the above equation becomes
\begin{align*}
\sum_{(k_1,k_2)\vdash ({\bf P}_1,{\bf P}_2)}&N^{\rm rel}[({\bf w}(k^1) , {\bf w}(k^2))]\\
&\cdot\prod^2_{i = 1}\prod^{l_i	}_{j=1}\prod_{w}\frac{(-1)^{k^i_{w,j}(w-1)}}{w^{k^i_{w,j}}}\frac{1}{m_w(k^i)!}\binom{m_w(k^i) - \sum^{j - 1}_{r = 1} k^i_{w, r}}{k^i_{w, j}}.
\end{align*}  
The result follows from the simple calculation
\begin{equation*}
\frac{1}{m_w(k^i)!}\binom{m_w(k^i)}{k^i_{w,1}}\binom{m_w(k^i)-k^i_{w,1}}{k^i_{w,2}}\ldots\binom{m_w(k^i)-\sum_{r=1}^{l_i-1}k^i_{w,r}}{k^i_{w,l_i}} = \prod^{l_i}_{j = 1}\frac{1}{k^i_{w,j}!}.
\end{equation*}
\end{proof}
Comparing the degeneration formula as rewritten in the Lemma with the MPS formula \eqref{mps}, we see that proving Theorem \ref{mpsdeg} is equivalent to establishing the identity
\begin{equation}\label{relative}
\chi(M^{\Theta_l-\rm st}_{(k^1,k^2)}(\mathcal{N})) = N^{\rm rel}[({\bf w}(k^1) , {\bf w}(k^2))]\prod^2_{i = 1}\prod^{l_i	}_{j=1}\prod_{w}w^{k^i_{w,j}}.
\end{equation}
In fact the right hand side of \eqref{relative} has a geometric interpretation as a suitable \emph{tropical count}. Here we will confine ourselves to the basic notions we need to state this equivalence, following \cite{gps} Section 2.1. 

Let $\overline{\Gamma}$ be a weighted, connected tree with only $1$-valent and $3$-valent vertices, thought of as a compact topological space in the canonical way. We remove the $1$-valent vertices to form the graph $\Gamma$. The noncompact edges are called \emph{unbounded} edges. We denote the induced weight function on the edges of $\Gamma$ by $w_{\Gamma}$. A \emph{parametrized rational tropical curve} in $\R^2$ is a proper map $h\!: \Gamma\to\R^2$ such that:
\begin{enumerate}
\item[$\bullet$] the restriction of $h$ to an edge is an embedding whose image is contained in an affine line of rational slope, and\\
\item[$\bullet$] a \emph{balancing condition} holds at the vertices. Namely, denoting by $m_i$ the primitive integral vector emanating from the image of a vertex $h(V)$ in the direction of an edge $h(E_i)$, we require 
\begin{equation*}
\sum^3_{i = 1} w_{\Gamma}(E_i) m_i = 0,
\end{equation*}  
where we are summing over all the edges which are adjecent to $V$. 
\end{enumerate}
A \emph{rational tropical curve} is the equivalence class of a rational parametrized tropical curve under reparametrizations which respect $w_{\Gamma}$. The \emph{multiplicity} at a vertex $V$ is defined as
\begin{equation*}
\mult_V(h) = w_{\Gamma}(E_1)w_{\Gamma}(E_2)|m_1\wedge m_2|,
\end{equation*}   
where by the balancing condition we can choose $E_1, E_2$ to be any two  edges adjecent to $V$. To total multiplicity of $h$ is then defined as 
\begin{equation*}
\mult(h) = \prod_V\mult_V(h).
\end{equation*}
Let us write $e_1, e_2$ for the versors of $\R^2$. A pair of weight vectors $({\bf w}_1, {\bf w}_2)$ encodes a tropical invariant, counting rational tropical curves $h$ which satisfy the following conditions:
\begin{enumerate}
\item[$\bullet$] the unbounded edges of $\Gamma$ are $E_{ij}$ for $1 \leq i \leq 2, 1 \leq j \leq t_i$, plus a single ``outgoing" edge $E_{\rm out}$. We require that $h(E_{ij})$ is contained in a line $e_{ij} + \R e_i$ for some \emph{prescribed} versors $e_{ij}$, and its unbounded direction is $-e_i$,\\
\item[$\bullet$] $w_{\Gamma}(E_{ij}) = w_{ij}$.
\end{enumerate}
Notice that the balancing condition implies that $h(E_{\rm out})$ lies on an affine line with direction $(|{\bf w}_1|, |{\bf w}_2|)$. The set of such tropical curves $h$ is finite, and it follows from the general theory (see e.g. \cite{gm}, \cite{mik}) that when we count curves $h$ taking into account the multiplicity $\mult(h)$ we get an integer $N^{\rm trop}[({\bf w}_1, {\bf w}_2)]$ which is \emph{independent} on the (generic) choice of displacements $e_{ij}$. The comparison result that we need, relating the Gromov-Witten invariants which appear in the degeneration formula to tropical counts, is then obtained by combining Theorems 3.4 and 4.4 in \cite{gps}:
\begin{equation}\label{equiv}
N^{\rm trop}[({\bf w}_1, {\bf w}_2)] = N^{\rm rel}[({\bf w}_1, {\bf w}_2)]\prod^2_{i = 1}\prod^{t_i	}_{j=1}w_{ij}.
\end{equation}
Thanks to the equivalence \eqref{equiv}, Theorem \ref{mpsdeg} follows from the following result:
\begin{pro}\label{eulgw} We have an equality of Euler characteristics and tropical counts
\begin{equation}\label{trop2chi}
N^{\rm trop}[({\bf w}(k^1) , {\bf w}(k^2))] = \chi(M^{\Theta_l-\rm st}_{(k^1,k^2)}(\mathcal{N})).
\end{equation}
\end{pro}
We will prove this equality using the scattering diagrams of \cite{gps} and Theorem 2.1 in \cite{poiss}. Notice however that in the special case when all the parts of the refinement $(k^1,k^2)$ equal $1$ the corresponding subquiver of $\mathcal{N}$ is isomorphic to $K(l_1, l_2)$, and \eqref{trop2chi} is an immediate consequence of the refined GW/Kronecker correspondence.

Let us we denote by $Q \subset \mathcal{N}$ the subquiver spanned by the support of (the dimension vector induced by) $(k^1, k^2)$. This is a complete bipartite quiver with $t_1$ sources and $t_2$ sinks. For each $w$, $Q$ contains $m_w(k^1)$ sources (respectively $m_w(k^2)$ sinks) with level $w$. We introduce the ring
\begin{equation}
R = \C[[x_{j_{(w',m')}}, y_{i_{(w,m)}} \mid w,w',m,m'\in\Nn]]   
\end{equation}  
with a Poisson bracket defined by
\begin{equation*}
\{x_{j_{(w',m')}}, y_{i_{(w,m)}}\} = \{j_{(w',m')}, i_{(w,m)}\}\,x_{j_{(w',m')}}\cdot y_{i_{(w,m)}} = w w'x_{j_{(w',m')}}\cdot y_{i_{(w,m)}}.
\end{equation*}
The Kontsevich-Soibelman Poisson automorphisms in this context are defined by
\begin{align*}
T_{j_{(w, m)}}(x_{j_{(w',m')}}) &= x_{j_{(w',m')}},\\
T_{j_{(w, m)}}(y_{i_{(w',m')}}) &= y_{i_{(w',m')}}\left(1 + x_{j_{(w, m)}}\right)^{w w'},
\end{align*}
and similarly
\begin{align*}
T_{i_{(w, m)}}(x_{j_{(w',m')}}) &= x_{j_{(w',m')}}\left(1 + y_{i_{(w, m)}}\right)^{-w w'},\\
T_{i_{(w, m)}}(y_{i_{(w',m')}}) &= y_{i_{(w',m')}}.
\end{align*}
According to Theorem 2.1 in \cite{poiss}, the product of operators
\begin{equation*}
\prod_{j_{(w, m)}\in Q_0} T_{j_{(w, m)}} \cdot \prod_{i_{(w',m')}\in Q_0} T_{i_{(w',m')}} 
\end{equation*}
can be expressed alternatively as a slope-ordered product $\prod^{\leftarrow}_{\mu\in\Q}T_{\mu}$, acting e.g. on the $y$ variables as
\begin{equation*}
T_{\mu}(y_{i_{(w,m)}}) = y_{i_{(w,m)}}\prod_{j_{(w',m')}\in Q_0} (Q_{\mu, j_{(w',m')}})^{w w'},
\end{equation*}
where we have denoted by $Q_{\mu, j_{(w',m')}}$ the generating series of Euler characteristics for moduli spaces of stable representations of $Q$ with slope $\mu$ and a $1$-dimensional framing at $j_{(w',m')}$. Recall however that we are only interested in the Euler characteristic $\chi(M^{\Theta_l-\rm st}_{(k^1,k^2)}(\mathcal{N}))$, i.e. for representations with dimension $1$ at each vertex. In this case framed representations coincide with ordinary representations, and we find that the coefficient of the monomial \begin{equation}\label{top}
\prod_{j_{(w',m')}\in Q_0} x_{j_{(w',m')}}\cdot\prod_{i_{(w,m)}\in Q_0}y_{i_{(w,m)}}
\end{equation} 
in the series $y^{-1}_{i_{(w,m)}}T_{\mu}(y_{i_{(w,m)}})$ is given by 
\begin{equation}\label{chiCoeff}
w \sum_{w'} w' m_{w'}(k^2)\,\chi(M^{\Theta_l-\rm st}_{(k^1,k^2)}(\mathcal{N}))
\end{equation}  
(where we choose $\mu$ to be the slope of the dimension vector induced by $(k^1, k^2)$). On the other hand we can compute the coefficient of \eqref{top} in a different way, by setting up an appropriate \emph{scattering diagram} in the sense of \cite{gps} Definition 1.2. To this end we need to identify $T_{i_{(w, m)}}, T_{j_{(w',m')}}$ with operators acting on the ring
\begin{equation*}
R' = \C[x^{\pm 1}, y^{\pm 1}][[\xi_{j_{(w,m)}}, \eta_{i_{(w',m')}} \mid w,w',m,m'\in\Nn]].
\end{equation*}
This is possible if we set
\begin{align}\label{identifications}
\nonumber x_{j_{(w, m)}}  &= \left(\xi_{j_{(w, m)}} x\right)^{w},\\
y_{i_{(w',m')}} &= \left(\eta_{i_{(w',m')}} y\right)^{w'},
\end{align}
from which 
\begin{align*}
T_{j_{(w, m)}}(\xi_{j_{(w', m')}} x) &= \xi_{j_{(w', m')}} x,\\
T_{j_{(w, m)}}(\eta_{i_{(w',m')}} y) &= \eta_{i_{(w',m')}} y\left(1 + \left(\xi_{j_{(w, m)}} x\right)^{w}\right)^{w},
\end{align*}
respectively
\begin{align*}
T_{i_{(w, m)}}(\xi_{j_{(w', m')}} x) &= \xi_{j_{(w', m')}} x\left(1 + \left(\eta_{i_{(w,m)}} y\right)^{w}\right)^{-w},\\
T_{i_{(w, m)}}(\eta_{i_{(w',m')}} y) &= \eta_{i_{(w',m')}} y.
\end{align*}
Then following the notation of \cite{gps} Section 0.1, we can make the identification 
\begin{equation*}
T_{j_{(w, m)}} = \theta_{(1, 0), \left(1 + \left(\xi_{j_{(w, m)}} x\right)^{w}\right)^{w}}
\end{equation*}
with a standard element of the tropical vertex group over $R'$, $\mathbb{V}_{R'}$,  and similarly
\begin{equation*}
T_{i_{(w, m)}} = \theta_{(0, 1), \left(1 + \left(\eta_{i_{(w,m)}} y\right)^{w}\right)^{w}}.
\end{equation*}
We are led to consider the saturated scattering diagram $\mathcal{S}$ (in the sense of \cite{gps} Section 1) for the product
 \begin{equation}\label{scatter1}
\prod_{j_{(w, m)}\in Q_0} \theta_{(1, 0), (1 + (\xi_{j_{(w, m)}} x)^{w})^{w}} \cdot \prod_{i_{(w',m')}\in Q_0} \theta_{(0, 1), (1 + (\eta_{i_{(w', m')}} y)^{w'})^{w'}}. 
\end{equation}
We only recall briefly that according to the general theory one starts with a generic configuration of horizontal lines $\mathfrak{d}_{j_{(w, m)}}$ in $\R^2$ (respectively vertical lines $\mathfrak{d}_{i_{(w', m')}}$), with attached weight functions $(1 + (\xi_{j_{(w, m)}} x)^{w})^{w}$ ($(1 + (\eta_{i_{(w', m')}} y)^{w'})^{w'}$ respectively). According to \cite{gps} Section 1.2 with a generic path $\gamma\!:[0, 1] \to \R^2$ one can associate an element $\theta_{\gamma} \in \mathbb{V}_{R'}$. The (essentially unique) saturated scattering diagram $\mathcal{S}$ is obtained by adding rays to the original configurations of lines to that for each closed loop the group element $\theta_{\gamma}$ becomes trivial (if at all defined), see \cite{gps} Theorem 1.4. The crucial point for us is that, thanks to the identifications \eqref{identifications}, $\mathcal{S}$ gives an alternative way of computing the ordered product factorization $\prod^{\leftarrow}_{\mu\in\Q}T_{\mu}$.

In fact since we are only interested in the coefficient of the monomial \eqref{top}, we are allowed to replace the ring $R'$ with its truncation
\begin{equation*}
\C[x^{\pm 1}, y^{\pm 1}][[\xi_{j_{(w,m)}}, \eta_{i_{(w',m')}} \mid w,w',m,m'\in\Nn]]/(\xi^{2w}_{j_{(w,m)}}, \eta^{2w'}_{i_{(w',m')}}),
\end{equation*}
and thus replace the scattering diagram for \eqref{scatter1} with the much simpler scattering diagram for the product
\begin{equation*}
\prod_{j_{(w, m)}\in Q_0} \theta_{(1, 0), 1 + w (\xi_{j_{(w, m)}} x)^{w}} \cdot \prod_{i_{(w',m')}\in Q_0} \theta_{(0, 1), 1 + w' (\eta_{i_{(w',m')}} y)^{w'}}.
\end{equation*}
Making the change of variables
\begin{equation*}
u_{j_{(w, m)}} = \xi^{w}_{j_{(w,m)}},\,\,\,v_{i_{(w', m')}} = \eta^{w'}_{i_{(w', m')}}
\end{equation*} 
we can as well consider the scattering diagram $\widetilde{\mathcal{S}}$ for the product
\begin{equation*}
\prod_{j_{(w, m)}\in Q_0} \theta_{(1, 0), 1 + w\,u_{j_{(w, m)}} x^{w}} \cdot \prod_{i_{(w',m')}\in Q_0} \theta_{(0, 1), 1 + w' v_{i_{(w',m')}} y^{w'}}.
\end{equation*}
over the ring 
\begin{equation*}
\widetilde{R'} = \C[x^{\pm 1}, y^{\pm 1}][[u_{j_{(w,m)}}, v_{i_{(w',m')}} \mid w,w',m,m'\in\Nn]]/(u^2_{j_{(w,m)}}, v^2_{i_{(w',m')}}).
\end{equation*}
According to \cite{gps} Theorem 2.4, there is a one to one correspondence between rays of $\widetilde{\mathcal{S}}$ and rational tropical curves for which the set $E_{ij}$ is contained in $\mathfrak{d}_{j_{(w,m)}}, \mathfrak{d}_{i_{(w', m')}}$ (so that the weight of a leg contained in $\mathfrak{d}_{j_{(w,m)}}$ is $w$, respectively $w'$ for $\mathfrak{d}_{i_{(w', m')}}$). What is more, if $f$ is a weight function containing the monomial \eqref{top}, it must have the form
\begin{equation*}
f = 1 + \mult(h) \prod_{j_{(w',m')}\in Q_0} u_{j_{(w',m')}} x^{w'} \cdot\prod_{i_{(w,m)}\in Q_0} v_{i_{(w,m)}} y^w 
\end{equation*}
where $h$ is the corresponding tropical curve. This is again a consequence of \cite{gps} Theorem 2.4. Indeed taking up for a moment the notation of \cite{gps} equation (2.1), in our case we have $w_{\rm out} = 1$ and the term $a_{i(\#J) q} \prod_{j \in J} u_{ij}$ vanishes except when $J = \{1\}$ and $q$ is one of our $w, w'$, so 
\begin{equation*}
\prod_{i, J, q} \left((\#J)! a_{i(\#J) q} \prod_{j \in J} u_{ij}\right) z^{m_{\rm out}} = \prod_{j_{(w',m')}\in Q_0} u_{j_{(w',m')}} x^{w'} \cdot\prod_{i_{(w,m)}\in Q_0} v_{i_{(w,m)}} y^w. 
\end{equation*}
Notice that the weight vector of $h$ is $({\bf w}(k^1), {\bf w}(k^2))$. Therefore the product of all such weight functions $f$ equals
\begin{equation*}
\Phi = 1 + N^{\rm trop}[({\bf w}(k^1), {\bf w}(k^2))] \prod_{j_{(w',m')}\in Q_0} u_{j_{(w',m')}} x^{w'} \cdot\prod_{i_{(w,m)}\in Q_0} v_{i_{(w,m)}} y^w 
\end{equation*}
Thanks to the choice of level structure on $\mathcal{N}$ (i.e. $l(i_{(w, m)}) = w, l(j_{(w', m')}) = w'$), the slope of each ray underlying one of the weight functions $f$ equals $\mu$, the slope of the dimension vector induced by $(k^1, k^2)$. By the uniqueness of ordered product factorizations in $\mathbb{V}_{\widetilde{R}'}$, we have that the coefficient of the monomial \eqref{top} in the series $y^{-1}_{i_{(w,m)}}T_{\mu}(y_{i_{(w,m)}})$ equals the nontrivial coefficient of the action of $\theta_{\Phi}$ on $v_{i_{(w, m)}} y^w$. Namely we have
\begin{equation*}
\theta_{\Phi}(v_{i_{(w, m)}} y^w) = \theta_{\Phi}(\eta^w_{i_{(w, m)}} y^w) = \eta^w_{i_{(w, m)}} y^w \cdot \Phi^{w \sum_{w'} w' m_{w'}(k^2)}, 
\end{equation*}   
which when expanded contains as the only nontrivial coefficient
\begin{equation}\label{tropCoeff}
w \sum_{w'} w' m_{w'}(k^2) N^{\rm trop}[({\bf w}(k^1), {\bf w}(k^2))].
\end{equation}
Our claim \eqref{trop2chi} follows by comparing \eqref{chiCoeff}, \eqref{tropCoeff}.\\ 
\textbf{Remark.} One can show (arguing by induction on $|p_{ij}|$) that the MPS formula and the coprime case of the refined GW/Kronecker correspondence imply the equality \eqref{trop2chi}.
\section{Euler characteristic via counting trees}\label{trees}
In this section we continue the investigation of the MPS formula $(\ref{mpsfor})$. Combined with localization techniques it implies that to calculate the Euler characteristic of moduli spaces it suffices to count trees. For a quiver $Q$ we denote by $\tilde{Q}$ its universal cover given by the vertex set 
\[\tilde{Q}_0=\{(q,w)\mid q\in Q_0,w\in W(Q)\}\]
and the arrow set
\[\tilde{Q}_1=\{\alpha_{(q,w)}:(q,w)\rightarrow (q',w\alpha)\mid\alpha:q\rightarrow q'\in Q_1\}.\]
Here $W(Q)$ denotes the set of words of $Q$, see \cite[Section 3.4]{wei} for a precise definition. Recall the localization theorem \cite[Corollary 3.14]{wei}:
\begin{thm} We have
\[\chi(M^{\Theta-\rm{st}}_d(Q))=\sum_{\tilde{d}}\chi(M^{\tilde{\Theta}-\rm{st}}_{\tilde{d}}(\tilde{Q})),\]
where $\tilde{d}$ ranges over all equivalence classes being compatible with $d$, and the slope function considered on $\tilde{Q}$ is the one induced by the slope function fixed on $Q$, i.e. we define the corresponding linear form $\tilde{\Theta}$ by $\tilde{\Theta}_{q'}=\Theta_q$ for all $q\in Q_0$ and for all $q'\in\tilde{Q}_0$ corresponding to $q$.
\end{thm}
We call a tuple $(\mathcal{Q},d)$ consisting of a finite subquiver $\mathcal{Q}$ of $\tilde{Q}$ and a dimension vector $d\in\Nn\mathcal{Q}_0$ localization data if $M^{\tilde{\Theta}-\rm{st}}_{d}(\mathcal{Q})\neq\emptyset$.

Even if the following machinery applies in a more general setting, we  concentrate on the quivers $K(l_1,l_2)$ and $\mathcal N$ respectively. We fix a partition $({\bf P}_1,{\bf P}_2)$ with $|{\bf P}_1|=d,\,|{\bf P}_2|=e$ and a refinement $(k^1,k^2)$. We consider the quiver $\mathcal{N}(k^1,k^2):=\mathrm{supp}(k^1,k^2)$ consisting of the full subquiver of $\mathcal{N}$ with vertices $\mathrm{supp}(k^1,k^2)_0=\{q\in\mathcal{N}_0\mid (k^1,k^2)_q\neq 0\}$. 
Since we have $(k^1,k^2)_q=1$ for all $q\in \mathcal{N}(k^1,k^2)_0$, every localization data $(T,(k^1,k^2))$ defines a subtree $T$ of $\mathcal{N}(k^1,k^2)$, i.e. a subquiver without cycles with at most one arrow between each two vertices.

In the following, fixing a subtree $T$ we denote by
$\beta$ the dimension vector and, moreover, we denote by $(d,e)$ the dimension type, i.e.
\[d:=\sum_{i\in T(I)}\beta_il(i)\text{ and }e:=\sum_{j\in T(J)}\beta_jl(j)\]
\begin{rem} \emph{Since for the dimension vectors $\beta\in\N T _0$ we are mostly interested in we have $\beta_q=1$ for all $q\in T_0$ and since $T$ is a tree, there exists only one stable representation $X$ up to isomorphism. In particular, we can assume that $X_{\alpha}=1$ for all $\alpha\in T_1$. Moreover, we have $M_{\beta}^{\tilde{\Theta}_l-\rm{st}}(T)=\{\rm pt\}$.}
\end{rem}

\noindent Since, in general, every connected tree with $n$ vertices has $n-1$ edges, every connected subtree of $\mathcal{N}(k^1,k^2)$ has $\sum_{i=1}^2\sum_{w}\sum_{j=1}^{l_i}k_{w,j}^i-1$ arrows.  Let $T(k^1,k^2)$ be the set of connected subtrees of $\mathcal{N}(k^1,k^2)$. For a tree $T\in T(k^1,k^2)$ and a subset $I'\subsetneq T(I)$ we define $\sigma_{I'}(T)=\sum_{j\in N_{I'}}l(j)$. Then $(T,\beta)$ with $\beta_q=1$ for all $q\in T_0$ is a localization data if and only if $\sigma_{I'}(T)>\frac{e}{d}|I'|$ for all $\emptyset\neq I'\subsetneq T(I)$ where $|I'|:=\sum_{i\in I'}l(i)$. Moreover since we have $\chi(M_{\beta}^{\tilde{\Theta}_l-\rm{st}}(T))=1$ it suffices to count such subtrees in order to calculate the Euler characteristic. More precisely defining
\[w(T)= \left\{
\begin{matrix}
1\textrm{ if } \sigma_{I'}(T)>\frac{e}{d}|I'|\text{ for all }\emptyset\neq I'\subsetneq T(I)\\
0\textrm{ otherwise} 
\end{matrix}\right\}\]
we obtain the following:
\begin{kor}\label{treeSum}
We have
$\chi(M^{\Theta_l-\rm{st}}_{(k^1,k^2)}(\mathcal{N}))=\sum_{T\in T(k^1,k^2)}w(T).$
\end{kor}
Notice that fixing a tree $T\in T(k^1,k^2)$ with $w(T)=1$ we can also forget about the colouring (but fix the level structure), and ask for the number of different embeddings into $\mathcal{N}(k^1,k^2)$. 

\section{A connection between localization data and tropical curves}\label{loctrop}
In this section we connect a recursive construction of tropical curves to a similar construction of localization data. This gives a possible recipe to obtain a direct correspondence between rational tropical curves and quiver localization data, as suggested by the equality of the respective counts Proposition \ref{eulgw}. We work out this correspondence in some examples in the following sections. In every case this construction can be used to compute the number of tropical curves and, therefore, the Euler characteristic of the corresponding moduli spaces recursively. 

\subsection{Recursive construction of curves}
The aim of this section is to construct tropical curves recursively. Therefore we first show that by choosing the lines $h(E_{ij})$ in a suitable  (but generic) way we can remove the last edge $E_{2,t_2}$, effectively decomposing one of our tropical curves into smaller ones. Moreover this construction works in the other direction as well. In particular we show that every tropical curve is obtained by glueing smaller ones. 

In the following we denote the coordinates of a vector $x\in\R^2$ by $(x^1,x^2)$. Let $h:\Gamma\rightarrow\R^2$ be a connected parametrized rational tropical curve with unbounded edges $E_{ij}$ of weights $w_{ij}$ and $E_{\rm out}$ for $1\leq i\leq 2$ and $1\leq j\leq t_i$. Let $h(E_{ij})$ be contained in the line $e_{ij}+\R e_i$. For every unbounded edge $E$ there exists a unique vertex $V$ such that $E$ is adjacent to $V$. We denote this vertex by $V(E)$. For every compact edge $E$ there exist two vertices $V_1(E)$ and $V_2(E)$ which are adjacent to $E$ where we assume that $h(V_2(E))^1>h(V_1(E))^1$. For a fixed tropical curve  $h:\Gamma\rightarrow\R^2$ we have $h(V_1(E))=h(V_2(E))+\lambda w_{\Gamma}(E) m$ for some $\lambda\in\R$ where $m$ denotes the primitive vector emanating from $h(V_1(E))$ in the direction of $h(V_2(E))$. The slope $\frac{w_{\Gamma}(E)m^2}{w_{\Gamma}(E)m^1}$ of $h(E)$ is abbreviated to $\mu(E)$ in the following. Note that it is important that the weight of $E$ is taken into account. \\

By the methods of \cite{gm}, see also \cite[Proposition 2.7]{gps}, we have that $N^{\rm trop}[{\bf w}_1,{\bf w}_2]$ does not depend on the (general) choice of the vectors $e_{ij}$. So we always assume that $e_{2,j}^1>e_{2,j-1}^1$ and $e_{1,j}^2>e_{1,j-1}^2$.

Let $F_1,\ldots,F_n$ be the edges such that there exists a point $(x^1_i,x^2_i)\in h(F_i)$ satisfying $e^1_{2,t_2-1}<x^1_i<e^1_{2,t_2}$ and, moreover, $h(V_2(E))^1\geq e_{2,t_2}^1$. Moreover, let $h(F_i)\subset \R f_{1,i}+f_{2,i}=:f_i$ and denote by $s_{i,j}\in\R^2$, $1\leq i<j\leq n$, the intersection point of $f_i$ and $f_j$ if there exists one. Again, by the methods of \cite{gm} we may assume that $s_{i,j}^1<e^1_{2,t_2}$ (by moving  $e_{2,t_2}$). This means that all intersection points of any two affine lines $f_i$ and $f_j$ lie on the left hand side of the line $e_{2,t_2}+\R e_2$. In particular we may assume that the edges $F_i$ are slope ordered, i.e. $\mu(F_i)\leq\mu(F_j)$ for $i<j$.

In the following, we say that a tropical curve satisfying these conditions is slope ordered. Then we have the following lemma:
\begin{lem}
We can choose the lines $h(E_{ij})$ in such a way that every tropical curve is slope ordered.
\end{lem}
In the following we assume that we have chosen the lines in this way and we also call such an arrangement of lines slope ordered. Let $\mu(F_i)=\frac{e_i}{d_i}$ with $e_i,d_i\in\N$ and $d_i> 0$. Clearly the vertex $V(E_{2,t_2})$ must be adjacent to $F_1$ and $G_0:=E_{2,t_2}$ and induces an edge $G_1$ with $\mu(G_1)=\frac{e_1+w_{2,t_2}}{d_1}$ such that $\frac{e_1+w_{2,t_2}}{d_1}>\frac{e_2}{d_2}$. With this notation in place, the following lemma can be proved by induction:
\begin{lem}
Let $h:\Gamma\rightarrow\R^2$ be a slope ordered tropical curve. Then there exist edges $G_1,\ldots, G_n$ and vertices $V_1,\ldots, V_n$ such that $V_i$ is adjacent to $F_{k}$, $G_{k-1}$ and $G_k$ and such that
$$\mu(G_k)=\frac{w_{2,t_2}+\sum_{i=1}^ke_i}{\sum_{i=1}^kd_i}>\frac{e_{k+1}}{d_{k+1}}$$
for $k=1,\ldots,n-1$. Moreover we have $G_n=E_{\rm out}$.
\end{lem}
\begin{rem} \emph{Note that for $w_{2,t_2}=1$ these conditions are part of the glueing conditions of \cite[Section 4.3]{wei}. The only missing property is the fourth one.}
\end{rem}
Fixing $(d,e)\in\mathbb{N}^2$ and $w\in\mathbb{N}$ we call a tuple $(d_i,e_i)_{i=1,\ldots,n}$ of pairs of natural numbers satisfying
\begin{eqnarray}\label{slope}w+\sum_{i=1}^ne_i=e,\,\sum_{i=1}^nd_i=d,\,\frac{e_i}{d_i}\leq\frac{e_{i+1}}{d_{i+1}}\text{ and }\frac{w+\sum_{i=1}^k e_i}{\sum_{i=1}^kd_i}>\frac{e_{k+1}}{d_{k+1}}\end{eqnarray}
a $w$-admissible decomposition of $(d,e)$. Obviously every slope ordered tropical curve defines a $w_{t_2}$-admissible decomposition of $(d,e)$.

We call a set partition $I_{\bullet}$ as introduced in Section \ref{4} proper if all parts are not empty. Fix a weight vector $({\bf w}_1,{\bf w}_2)$ with $d=\sum_{j=1}^{t_1}w_{1j}$ and $e=\sum_{j=1}^{t_2}w_{2j}$ satisfying $w_{ij}\leq w_{i(j+1)}$. Every $w_{2,t_2}$-admissible decomposition of $(d,e)$ defines two ordered partitions of $d$ and $e-w_{2,t_2}$ respectively. Then every tuple of set partitions of $\{1,\ldots,t_1\}$ and $\{1,\ldots,t_2-1\}$ respectively which is compatible with $(d_i,e_i)_{i=1,\ldots,n}$  defines $n$ tuples of weight vectors $({\bf w}(i)_1,{\bf w}(i)_2)$ with $d_i=\sum_{j}{\bf w}(i)_{1j}$ and $e_i=\sum_j{\bf w}(i)_{2j}$. 

Now let $h_i:\Gamma_i\rightarrow\R^2$, $i=1,\ldots,n$, be tropical curves with unbounded edges corresponding to the weights $({\bf w}(i)_1,{\bf w}(i)_2)$ and with outgoing edges $E_{i,{\rm out}}$. Moreover let $h(E_{2,t_2})\subseteq e_{2,t_2}+\R e_2$ be the embedding of $E_{2,t_2}$. We may assume that for all intersection points $s_{ij}$ of the affine lines $f_i=\R f_{1,i}+f_{2,i}$ containing $h(E_{i,{\rm out}})$ we have $s_{ij}^1<e_{2,t_2}^1$. We may assume that $f_{1,i}^1\geq 0$. In particular, we have $\mu(E_{i,{\rm out}})\leq\mu(E_{i+1,{\rm out}})$.

Therefore, these curves recursively define a tropical curve $h:\Gamma\rightarrow\R^2$ in the following way: $h(E_{1,{\rm out}})$ and $h(E_{2,t_2})$ have a unique intersection point $s_1$. Thus to $\Gamma$ we add a vertex $V_1$ with $h(V_1)=s_1$ and an unbounded edge $F_1$ with adjacent vertex $V_1$ setting $h(F_1)=\{h(V_1)+\R^{\geq 0}((d_1,e_1)+(0,w_{2,t_2}))\}$. Additionally we bind $E_{1,{\rm out}}$ by $V_1$ and so we modify its image in an appropriate way. In general, since $(d_i,e_i)_i$ is a $w_{t_2}$-admissible decomposition, there exists an intersection point $s_{j+1}$ of $h(F_j)$ and $h(E_{j+1,{\rm out}})$. Thus we add a vertex $V_{j+1}$ with $h(V_{j+1})=s_{j+1}$ and an unbounded edge $F_{j+1}$ with adjacent vertex $V_{j+1}$ setting $h(F_{j+1})=\{h(V_{j+1})+\R^{\geq 0}(\sum_{i=1}^{j+1}(d_i,e_i)+(0,w_{2,t_2}))\}$. As above we bind $E_{j+1,{\rm out}}$ and $F_{j}$ by $V_{j+1}$ and we again modify their images appropriately.\\

Considering all the $w_{t_2}$-admissible decompositions and all the sets of tropical curves (embedded in the chosen line arrangement) as above at once we can assume that $s_{ij}^1<e_{2,t_2}^1$ for all possible intersection points. So we have the following:
\begin{thm}\label{troprec}
\begin{eqnarray*}N^{\rm trop}({\bf w}_1,{\bf w}_2)&=&\sum_{(d_i,e_i)_i}\sum_{I_{\bullet}}\prod_{i=1}^nN^{\rm trop}[{\bf w}(i)_1,{\bf w}(i)_2]\\&&\left|\prod_{k=1}^{n}(e_{k}\sum_{i=1}^{k-1}d_i-d_{k}(\sum_{i=1}^{k-1}e_i+w_{2,t_2}))\right|\end{eqnarray*}
where we first sum over all $w_{t_2}$-admissible decompositions of $(d,e)$ and then over all proper set partitions $I_{\bullet}$ which are compatible with the partitions $(d,e-w_{2,t_2})=(\sum_{i=1}^nd_i,\sum_{i=1}^ne_i)$.
\end{thm}
{\it Proof.} We just need to determine the multiplicities of the vertices where the original curves are glued. For their multiplicities we get
\[{\rm Mult}_{V_{k}}(h)=\left|\begin{pmatrix}\sum_{i=1}^{k-1}d_i&d_k\\\sum_{i=1}^{k-1}e_i+w_{2,t_2}&e_k\end{pmatrix}\right|\]
for $k=1,\ldots,n$.
\qed\\
\subsection{Recursive construction of localization data}
On the quiver side we have a similar construction:
let $(\mathcal{Q}_1,\beta_1),\ldots, (\mathcal{Q}_n,\beta_n)$ be semistable consisting of disjoint subquivers of $\mathcal{N}$ and a dimension vector $\beta_i\in\mathbb{N}(\mathcal{Q}_i)_0$ of type one (namely $(\beta_i)_q=1$ for all $q\in(\mathcal{Q}_i)_0$), of dimension type $(d_i,e_i)$, i.e. we have
\[\sum_{q\in\mathcal{Q}_i(I)}l(q)=d_i
\text{ and }\sum_{q\in\mathcal{Q}_i(J)}l(q)=e_i,\]
 and $M^{\Theta_l-\rm sst}_{\beta_i}(\mathcal{Q}_i)\neq\emptyset$. Moreover let
the tuple $(d_i,e_i)_{i=1,\ldots,n}$ be a $w$-admissible decomposition of $(d,e)$ where $e:=w+\sum_{i=1}^n e_i$ and $d:=\sum_{i=1}^nd_i$. Consider the tuple $(\mathcal{Q},\beta)$ consisting of the quiver $\mathcal{Q}$ defined by the vertices $\mathcal{Q}_0=\bigcup_{j=1}^n(\mathcal{Q}_j)_0\cup\{q\}$ with $l(q)=w$ and the arrows $\mathcal{Q}_1=\bigcup_{j=1}^n(\mathcal{Q}_j)_1\cup\{\alpha:i\rightarrow q\mid i\in\mathcal{Q}_j(I), j=1,\ldots,n\}$ and the dimension vector $\beta$ obtained by setting $\beta_q=1$.\\

\begin{lem}
Let $(d_i,e_i)_i$ be a $w$-admissible decomposition of $(d,e)$ and let $I\subsetneq\{1,\ldots,n\}$. Then we have
\begin{enumerate}
\item \[\frac{\sum_{i\in I}e_i}{\sum_{i\in I}d_i}\leq\frac{e_n}{d_n}\]
\item $$\frac{w+\sum_{i\in I}e_i}{\sum_{i\in I}d_i}>\frac{w+\sum_{i=1}^ne_i}{\sum_{i=1}^nd_i}$$
\end{enumerate}\end{lem}
{\it Proof.}
Let $I_{\max}\in I$ the largest number in $I$. We proceed by induction on $n$. If $n\in I$, the first inequality is equivalent to
\[\frac{\sum_{i\in I\backslash\{n\}}e_i}{\sum_{i\in I\backslash\{n\}}d_i}\leq\frac{e_n}{d_n}.\]
By induction hypothesis we have
\[\frac{\sum_{i\in I\backslash\{n\}}e_i}{\sum_{i\in I\backslash\{n\}}d_i}\leq\frac{e_{(I\backslash\{n\})_{\max}}}{d_{(I\backslash\{n\})_{\max}}}\leq\frac{e_{n-1}}{d_{n-1}}\leq\frac{e_n}{d_n}.\]
If $n\notin I$, we can apply the induction hypothesis.\\

In order to prove the second inequality, we first assume that $n\notin I$. Then we proceed by induction on $n$. It is easy to check that
\[\frac{w+\sum_{i=1}^{n-1}e_i}{\sum_{i=1}^{n-1}d_i}>\frac{w+\sum_{i=1}^ne_i}{\sum_{i=1}^nd_i}\Leftrightarrow\frac{w+\sum_{i=1}^{n-1}e_i}{\sum_{i=1}^{n-1}d_i}>\frac{e_n}{d_n}.\]
Moreover by the induction hypothesis we have
\[\frac{w+\sum_{i\in I}e_i}{\sum_{i\in I}d_i}>\frac{w+\sum_{i=1}^{n-1}e_i}{\sum_{i=1}^{n-1}d_i}.\]
If $n\in I$, by the first statement we have
$\frac{\sum_{i\in I'}e_i}{\sum_{i\in I'}d_i}\leq\frac{e_n}{d_n}$  where $I':=\{1,\ldots,n\}\backslash I$. Since the second inequality of the statement is equivalent to 
$$\frac{w+\sum_{i\in I}e_i}{\sum_{i\in I}d_i}>\frac{\sum_{i\in I'}e_i}{\sum_{i\in I'}d_i},$$
it suffices to show that
$$\frac{w+\sum_{i\in I}e_i}{\sum_{i\in I}d_i}>\frac{e_n}{d_n}.$$
Since this is equivalent to
$$\frac{w+\sum_{i\in I\backslash\{n\}}e_i}{\sum_{i\in I\backslash\{n\}}d_i}>\frac{e_n}{d_n},$$ we can apply the induction hypothesis using $\frac{w+\sum_{i=1}^{n-1}e_i}{\sum_{i=1}^{n-1}d_i}>\frac{e_n}{d_n}$.\qed

\begin{thm}\label{glue}
The tuple $(\mathcal{Q},\beta)$ is stable. In particular, every stable torus fixed point of this quiver defines a stable torus fixed point of $\mathcal{N}$ of type $(d,e)$. 
\end{thm}
{\it Proof.} Since we have $\beta_q=1$ for all $q\in\mathcal{Q}_0$, we consider the representation $X$ defined by $X_{\alpha}=1$ for all $\alpha\in\mathcal{Q}_1$. Let $I_k\subseteq \mathcal{Q}_k(I)$ be arbitrary subsets with $t_k:=\sum_{i\in I_k}l(i)$ and let $s_k:=\sigma_{I_k}$ such that $I_k\neq \mathcal{Q}(I)_k$ for at least one $k$ and $I_k\neq\emptyset$ for at least one $k$. We have to show
\[w+\sum_{i=1}^ns_i>\frac{w+\sum_{i=1}^ne_i}{\sum_{i=1}^nd_i}(\sum_{i=1}^nt_i).\]
Since the tuples we started with are semistable, we have
\[w+\sum_{i=1}^ns_i\geq w+\sum_{i=1}^n\frac{e_i}{d_i}t_i.\]
Thus it suffices to show
\[w>\sum_{j=1}^nt_j\frac{d_j(w+\sum_{i=1}^ne_i)-e_j(\sum_{i=1}^nd_i)}{d_j(\sum_{i=1}^nd_i)}.\]
Therefore, if
\[k_j:=\frac{d_j(w+\sum_{i=1}^ne_i)-e_j(\sum_{i=1}^nd_i)}{d_j(\sum_{i=1}^nd_i)}> 0,\] we can assume that $t_j=d_j$, and otherwise we can assume that $t_j=0$.
Let $I\subseteq\{1,\ldots,n\}$ such that $j\in I\Leftrightarrow k_j>0$. Then we have to show that
$$\frac{w+\sum_{i\in I}e_i}{\sum_{i\in I}d_i}>\frac{w+\sum_{i=1}^ne_i}{\sum_{i=1}^nd_i}$$
what follows by the preceding lemma. Note that if $I=\{1,\ldots,n\}$, i.e. $k_j>0$ for all $1\leq j\leq n$, from $w=\sum_{j=1}^nd_jk_j$ it follows that
$w>\sum_{j=1}^nt_jk_j$ if $t_j<d_j$ for at least one $j$.\\


\qed

\begin{rem}
\end{rem}
\begin{itemize}
\item It would be interesting to know if every localization data can be obtained by this construction. We conjecture that this is true, but it seems more difficult to prove this than the tropical analogue.
\item The main goal we have in mind is to construct a direct correspondence between tropical curves and localization data. Given a tropical curve of slope $(d,e)$, say with multiplicity $m$, there should be $m$ localization data of dimension type $(d,e)$  corresponding to this tropical curve. We see the two recursive constructions which we have described as a step in this direction: they give us a way to glue smaller objects in order to build more complicated ones. Moreover on both sides we have the same numerical conditions. So starting with smaller objects for which a correspondence is known this should give a correspondence between the glued objects.
\item In some cases such a correspondence is obtained immediately: assume that we have decomposed $(d,e)=(d_s,e_s)+(d',e')$ as in \cite[Section 4.3]{wei}. Then we have $|d_se-de_s|=1$. In particular, the preceding methods give a one-to-one correspondence in this case. Indeed, we can understand the vertex corresponding to the last leg as the glueing vertex.
\item Unfortunately, the construction does not always give a canonical correspondence. Consider the data
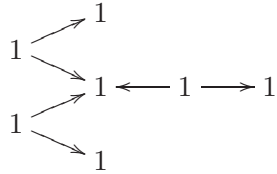
\[
\begin{xy}
\xymatrix@R0.5pt@C20pt{
&1\\1\ar[ru]\ar[rd]&\\&1&1\ar[l]\ar[r]&1\\1\ar[ru]\ar[rd]&\\&1 }
\end{xy}
\]
with a {\it fixed} colouring. Glueing an additional sink, we get the data 
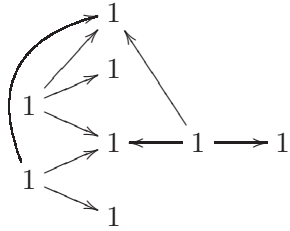
\[
\begin{xy}
\xymatrix@R0.5pt@C20pt{
&1&&\\\\&1\\1\ar[ruuu]\ar[ru]\ar[rd]&\\&1&1\ar[luuuu]\ar[l]\ar[r]&1\\1\ar@/^1.9pc/[ruuuuu]\ar[ru]\ar[rd]&\\&1 }
\end{xy}
\]
which has six subdata defining localization data, obtained by deleting two arrows in an appropriate way. 

But on the tropical curve side we only get three new curves in this way because we glue a curve of slope $(3,4)$ and one of slope $(0,1)$. This already gives the impression that we should not consider all possible localization data of the constructed stable tuples (with cycles as above). 
\item Notice that Corollary \ref{troprec} also gives a recursive formula for the Euler characteristic of moduli spaces. Indeed, even for $w$-admissible decompositions of $(d,e)$ involving non-primitive vectors $(d_i,e_i)$ one ends up with primitive vectors after finitely many steps.
\end{itemize}

\subsection{Examples and discussion}\label{examples}
In this section we discuss two examples in which we obtain a direct correspondence between tropical curves and localization data by using the methods described above.
\subsubsection{The case $(2,2n+1)$}
We consider the example of $2n+1$ points in the projective plane, i.e. $({\bf P}_1,{\bf P}_2)=(2,1^{2n+1})$. There exist two refined partitions which are the partition itself and $(k^1,k^2)=(1+1,1^{2n+1})$. In the first case, the only tree to consider is 
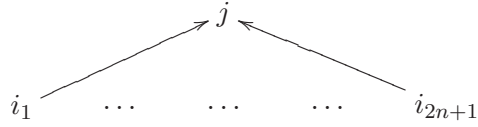
\[
\begin{xy}
\xymatrix@R2pt@C20pt{
&&j\\\\\\i_1\ar[rruuu]&\dots&\dots&\dots&i_{2n+1}\ar[lluuu]}
\end{xy}
\]
with $l(j)=2$ and $l(i_k)=1$. In the second case, the only tree to consider is
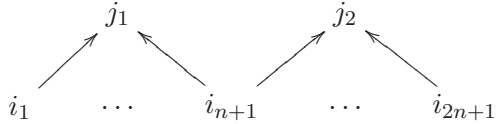
\[
\begin{xy}
\xymatrix@R2pt@C20pt{
&j_1&&j_2&\\\\\\i_1\ar[ruuu]&\dots&i_{n+1}\ar[luuu]\ar[ruuu]&\dots&i_{2n+1}\ar[luuu]}
\end{xy}
\]
Now it is easy to check that we have $2^{2n+1}$ different embeddings (or colourings) in the first case and 
\[\binom{2n+1}{n}\binom{n+1}{n}\] different embeddings in the second case. Thus by the MPS formula for the Euler characteristic we get:
\begin{eqnarray}\label{ec}\chi(M^{\Theta-\mathrm{st}}(2,1^{2n+1}))=\frac{1}{2}\binom{2n+1}{n}\binom{n+1}{n}-\frac{1}{4}2^{2n+1}.\end{eqnarray}
Following the construction of the last section, we have to decompose the vector $(2,2n+1)$ into a $1$-admissible tuple $(d_i,e_i)$. The only two possibilities are
\[(d_1,e_1)=(d_2,e_2)=(1,n)\text{ and }(d_1,e_1)=(2,2n),\]
and, moreover, the only $1$-admissible decomposition of $(2,2n)$ is $(d_1,e_1)=(2,2n-1)$.
In order to get a direct correspondence, we can proceed as follows: for the first $1$-admissible decomposition, the construction is straightforward. We just pick the two corresponding localization data of type $(1,n)$ and glue them in $i_{2n+1}$. 
In the second case assume that we have already constructed the curves corresponding to $(2,2n-1)$. The only way to obtain a curve corresponding to $(2,2n+1)$ from such a curve is to glue twice a curve of slope $(0,1)$ to it. If $m$ is the multiplicity of the tropical curve of slope $(2,2n-1)$, the multiplicity of the resulting curve is $4m$. 
On the quiver side this means that we have to construct four localization data of type $(2,2n+1)$ from every localization data of type $(2,2n-1)$. Consider the uncoloured localization data
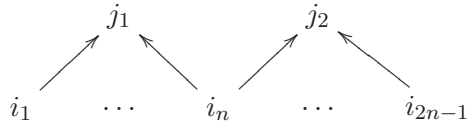
\[
\begin{xy}
\xymatrix@R2pt@C20pt{
&j_1&&j_2&\\\\\\i_1\ar[ruuu]&\dots&i_{n}\ar[luuu]\ar[ruuu]&\dots&i_{2n-1}\ar[luuu]}
\end{xy}
\]

Considering the construction of the last section we can construct a stable data of type $(2,2n)$ starting with this one. This leads to two semistable tuples by deleting one of the two new arrows. In short, we just glue the vertex $i_{2n}$ to one of the sinks. By the last section we now have to consider the following tuple
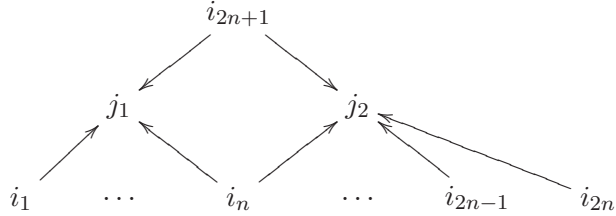
\[
\begin{xy}
\xymatrix@R2pt@C20pt{&&i_{2n+1}\ar[rddd]\ar[lddd]\\\\\\
&j_1&&j_2&\\\\\\i_1\ar[ruuu]&\dots&i_{n}\ar[luuu]\ar[ruuu]&\dots&i_{2n-1}\ar[luuu]&i_{2n}\ar[lluuu]}
\end{xy}
\]
But now it is easy to check that we have two possibilities to obtain a localization data from this.
On the curve side we consider a line arrangement of the following shape:\\
\setlength{\unitlength}{1cm}
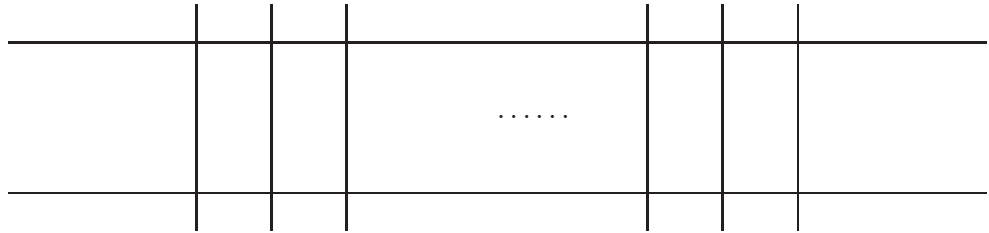
\begin{picture}(20,3.5)(1,0)
\linethickness{.2mm}
\multiput(1,0.5)(0,2){2}
{\line(1,0){13}}
\multiput(3.5,0)(1,0){3}
{\line(0,1){3}}
\put(7.5,1.5){\dots\dots}
\multiput(9.5,0)(1,0){3}
{\line(0,1){3}}
\end{picture}\vspace{0.3cm}
with $2n+1$ vertical legs. In order to determine the corresponding tropical curves, we first consider the tropical curve of weight one for the partition $(1,1^{n})$ (here for $n=3$), i.e.:\\
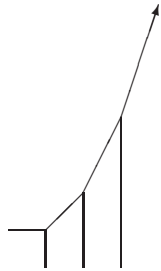
\begin{picture}(20,4)(-6,0)
\linethickness{.2mm}
\put(-1,0.5){\line(1,0){0.5}}
\put(-0.5,0){\line(0,1){0.5}}
\put(0,0){\line(0,1){1}}
\put(0.5,0){\line(0,1){2}}
\put(-0.5,0.5){\line(1,1){0.5}}
\put(0,1){\line(1,2){0.5}}
\put(0.5,2){\vector(1,3){0.5}}
\end{picture}\vspace{0.3cm}
For general $n$ we have $\binom{2n}{n}$ possibilities to embed the curves corresponding to the partition $(1,1^{n})$ into the upper row of the line arrangement above and another curve of the same slope into the lower row. Then we can glue these two curves as described in the last section. 
For the other refined partition which is the partition itself we obviously get one curve of weight $2^{2n+1}$. To sum up, we get
$N^{\rm trop}(2,1^{2n+1})=\frac{1}{2}(\binom{2n}{n}+4\binom{n}{n-1}\binom{2n-1}{n-1})-\frac{1}{4}2^{2n+1}$ which is easily seen to be the same as the expression (\ref{ec}).
For $n=1$ we get the following localization data and the following curves of multiplicity four, one and one respectively:\\
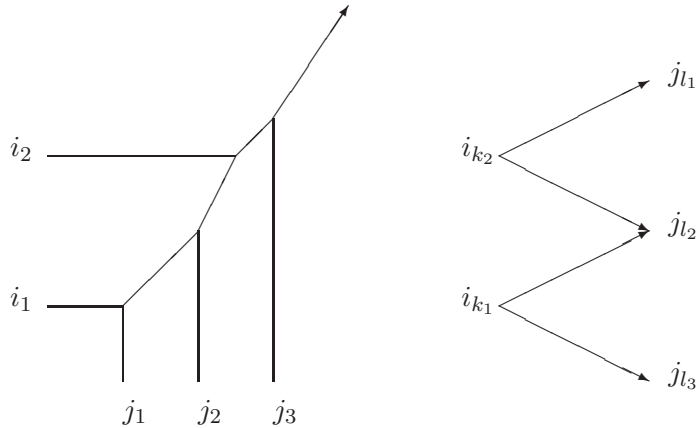
\begin{picture}(20,5)(-2.5,0)
\linethickness{.2mm}
\put(-2,1){\line(1,0){1}} 
\put(-2.5,1){$i_1$}
\put(-2,3){\line(1,0){2.5}}
\put(-2.5,3){$i_2$}
\put(-1,0){\line(0,1){1}}
\put(-1,-0.5){$j_1$}
\put(0,0){\line(0,1){2}}
\put(0,-0.5){$j_2$}
\put(1,0){\line(0,1){3.5}}
\put(1,-0.5){$j_3$}
\put(-1,1){\line(1,1){1}}
\put(0.5,3){\line(1,1){0.5}}
\put(0,2){\line(1,2){0.5}}
\put(1,3.5){\vector(2,3){1}}
\put(3.5,1){$i_{k_1}$}
\put(4,1){\vector(2,1){2}}
\put(4,1){\vector(2,-1){2}}
\put(3.5,3){$i_{k_2}$}
\put(4,3){\vector(2,1){2}}
\put(4,3){\vector(2,-1){2}}
\put(6.25,0){$j_{l_3}$}
\put(6.25,2){$j_{l_2}$}
\put(6.25,4){$j_{l_1}$}
\end{picture}\\\\\\
with $l_1,l_2\in\{1,2\}$ and $k_1,k_2\in\{1,2\}$ and $l_3=3$ which are four localization data. For the curves\\
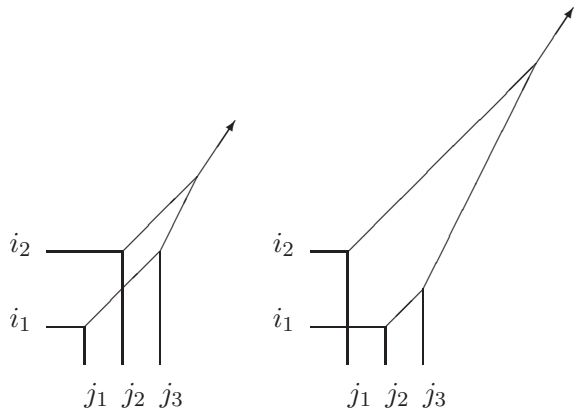
\begin{picture}(10,5)(-2.5,0)
\linethickness{.2mm}
\put(-1,0.5){\line(1,0){0.5}}
\put(-1.5,0.5){$i_1$}
\put(-1,1.5){\line(1,0){1}}
\put(-1.5,1.5){$i_2$}
\put(-0.5,0){\line(0,1){0.5}}
\put(-0.5,-0.5){$j_1$}
\put(0,0){\line(0,1){1.5}}
\put(0,-0.5){$j_2$}
\put(0.5,0){\line(0,1){1.5}}
\put(0.5,-0.5){$j_3$}
\put(-0.5,0.5){\line(1,1){1}}
\put(0,1.5){\line(1,1){1}}
\put(0.5,1.5){\line(1,2){0.5}}
\put(1,2.5){\vector(2,3){0.5}}

\put(2.5,0.5){\line(1,0){1}}
\put(2.0,0.5){$i_1$}
\put(2.5,1.5){\line(1,0){0.5}}
\put(2.0,1.5){$i_2$}
\put(3,0){\line(0,1){1.5}}
\put(3,-0.5){$j_1$}
\put(3.5,0){\line(0,1){0.5}}
\put(3.5,-0.5){$j_2$}
\put(4,0){\line(0,1){1}}
\put(4,-0.5){$j_3$}
\put(3.5,0.5){\line(1,1){0.5}}
\put(3,1.5){\line(1,1){2.5}}
\put(4,1){\line(1,2){1.5}}
\put(5.5,4){\vector(2,3){0.5}}
\end{picture}\\\\\\
we get the same quiver coloured by $k_1=1,\,k_2=2,\,l_1=2,\,l_2=3,\,l_3=1$ and
$k_1=1,\,k_2=2,\,l_1=1,\,l_2=3,\,l_3=2$ respectively.
%
%
\subsubsection{The case $(d,d+1)$}
We consider the dimension vector $(d,d+1)$ concentrating on the trivial refinement $(1^d,1^{d+1})$. The $1$-admissible decompositions of $(d,d+1)$ are given by $(d,d+1)=(0,1)+\sum_{i=1}^n(d_i,d_i)$ for some $n\leq d$. Moreover every slope-ordered tropical curve of slope $(d,d)$ is obtained by a tropical curve of slope $(d-1,d)$ glued with one of slope $(1,0)$. 
So for fixed tropical curves/localization data of slope $(d_i-1,d_i)$  we have to glue them in a certain way in order to get new tropical curves/localization data. If the multiplicities of these tropical curves are $m_1,\ldots,m_n$ the multiplicity of the new curve is easily determined to be
\[\prod_{i=1}^nm_id_i^2.\]
On the quiver side this means that we have to construct $\Pi_{i=1}^nd_i^2$ new localization data from those of type $(d_i-1,d_i)_i$.
By the results of \cite[Section 6.2]{wei} we know that every source of a localization data of dimension type $(d,d+1)$ of type one has exactly two neighbours and, therefore, is obtained by glueing the following data
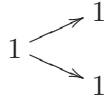
\[
\begin{xy}
\xymatrix@R0.5pt@C20pt{
&1\\1\ar[ru]\ar[rd]&\\&1}
\end{xy}
\]
and colouring the vertices. Thus fix an $n$-tuple of localization data $(\mathcal{Q}_i,\beta_i)$ of type $(d_i-1,d_i)$ with a fixed embedding into $\mathcal{N}$. Let $R_i\subset\mathcal{Q}_i(I)\times\mathcal{Q}_i(J)$ be the arrows of $\mathcal{Q}_i$. By \cite[Section 6.2]{wei} it is known that every connected subdata of dimension type $(d_i,d_i+1)$ is a localization data of this dimension type. So we may restrict to the case $n=1$. We are interested in certain semistable subtuples of type $(d_i,d_i)$ such that every source has at most two neighbours. Fixed such a subtuple there is exactly one possibility to glue an additional sink in order to get a localization data of type $(d_i,d_i+1)$ which is a subdata of the one constructed in Theorem \ref{glue}. We proceed as follows:
Let $\mathcal{Q}'(I):=\mathcal{Q}(I)\cup\{i_d\}$ and $\mathcal{Q}'(J)=\mathcal{Q}(J)$. Now there are several possibilities for the arrows. Initially, we consider $R':=R\cup\{(i_d,j)\}$ for some $j\in\mathcal{Q}'(J)$. Note that this gives us $d$ choices. Secondly, we consider the arrows given by
\[R':=(R\cup\{(i_d,j_1),(i_d,j_2)\})\backslash\{(i,j_k)\}\]
for some $j_1,j_2\in \mathcal{Q}'(J)$ with $j_1\neq j_2$, $i\in N_{j_k}$ (with $i\neq i_d$) for $k\in\{1,2\}$. This gives $2\binom{d}{2}$ choices. Note that
$d+2\binom{d}{2}=d^2$.

\begin{thm}
By this construction we get $\prod_{i=1}^nd_i^2$ localization data of type $(d_i,d_i+1)$ starting with localization data of type $(d_i-1,d_i)$ for $i=1,\ldots,n$. Moreover every localization data is obtained in this way.
\end{thm}
{\it Proof.} Consider a localization data of type $(d,d+1)$. By deleting the vertex $j_{d+1}$ (including the corresponding arrows) we get $n$ semistable subdata of type $(d_i,d_i)$ for $i=1,\ldots, n$. Let $q(i)_{\max}\in\mathcal{Q}_i(I)$ be the source with the maximal index. If $|N_{q(i)_{\max}}|=1$ we also delete this vertex and get a localization data of type $(d_i-1,d_i)$. If $|N_{q(i)_{\max}}|=2$ there exists exactly one source $q(i)\in\mathcal{Q}_i(I)$ such that $|N_{q(i)}|=1$. After deleting $q(i)_{\max}$, there exists one possibility to obtain a localization data by adding an extra arrow $(q(i),j)$ where $j\in N_{q(i)_{\max}}$. This already shows that every localization data of type $(d,d+1)$ is obtained by this construction. \qed

This is enough to describe the required correspondence between localization data and curves in this case.

\vspace{.5cm}
Bergische Universit\"at Wuppertal\\
reineke@math.uni-wuppertal.de
\vspace{.5cm}\\
Trinity College, Cambridge and Universit\`a di Pavia\\
js807@cam.ac.uk / jacopo.stoppa@unipv.it
\vspace{.5cm}\\
Bergische Universit\"at Wuppertal\\
weist@math.uni-wuppertal.de
\end{document}